\documentclass[11pt]{amsart}

\usepackage{graphpap}
\usepackage{amssymb, color}
\usepackage{amsfonts}
\newtheorem{theo}{Theorem}
\newtheorem{theoreme}{Theorem}[section]

\newtheorem{proposition}{Proposition}
\newtheorem{lem}[proposition]{Lemma}
\newtheorem{corollaire}[proposition]{Corollary}
\newtheorem{defi}[proposition]{Definition}
\newtheorem{remarque}[proposition]{Remark}

\numberwithin{equation}{section}
\numberwithin{proposition}{section}

\def\rit{{\mathbb R}}

\def\tit{{\mathbb T}}

\def\eps{\varepsilon}

\def\beq{\begin{equation}}
\def\eeq{\end{equation}}

\def\Im{\textrm{Im}}
\def\Re{\textrm{Re}} 
\def\11{{\rm 1~\hspace{-1.4ex}l} }
\def\R{\mathbb R}
\def\C{\mathbb C}
\def\Z{\mathbb Z}
\def\N{\mathbb N}

\def\T{\mathbb T}

\begin{document}
\title[Transverse instability]{Transverse nonlinear instability  of solitary waves for some Hamiltonian PDE's}
\author{Frederic Rousset}
\address{
Laboratoire J.A. Dieudonn\'e, Universit\'e de Nice, 06108 Nice cedex 2, France }
\email{frederic.rousset@unice.fr}
\author{Nikolay Tzvetkov}
\address{D\'epartement de Math\'ematiques, Universit\'e Lille I, 59 655 Villeneuve d'Ascq cedex, France}
\email{nikolay.tzvetkov@math.univ-lille1.fr}
\date{}
\begin{abstract} 
We present a general result  of transverse nonlinear instability of 1-d
solitary waves  for  
Hamiltonian PDE's for both periodic 
or localized transverse perturbations. Our main structural assumption is that the linear part of the $1d$ 
model and the transverse perturbation ``have the same sign''. Our result applies to the generalized 
KP-I equation,  the Nonlinear Schr\"odinger equation,  the generalized Boussinesq system and  the
Zakharov-Kuznetsov equation and we hope that it may be useful in other contexts.
\end{abstract}

\maketitle
\tableofcontents
%

\section{Introduction}
 A lot of two-dimensional dispersive equations  possess  one-dimensional
  solitary waves which are stable when submitted to one-dimensional 
   perturbations  but which are destabilized  when submitted to general
     two-dimensional perturbations. This phenomenon has been known
      for a long time in the physics literature. For example,  by using
       the Lax pair structure of the KP-I equation, it was proven in \cite{ZK}
        that the KdV solitary wave seen as a 1d solution of the KP-I equation
         is unstable. For non-integrable equations, the general instability
   theory of solitary waves of \cite{GSS} does not seem to apply  since
    the 1-d solitary wave is not a constrained  critical point of the Hamiltonian of the 2d
     equation. Nevertheless, in some cases, the linear instability 
      can be proven by some simple bifurcation arguments, 
       for example,   the  linear instability of the 1d  solitary wave of
        the 2d  Nonlinear   Schr\"odinger equation (NLS)  can be proven
         by the   Zakharov-Rubenchik bifurcation argument for small 
transverse frequencies. Consequently, it seems interesting to  reduce
 the proof of nonlinear instability to the search for unstable eigenmode
  for the linearized equation by 
 proving that   linear instability implies  nonlinear
  instability for a large class of equations.  
      
In \cite{RT},  we have shown that
 the method developed by Grenier  \cite{Grenier} for the incompressible Euler equation can be adapted
  to prove transverse instability of solitary waves in dispersive models.
   More precisely, we have proven 
  two nonlinear instability results for  solitary waves  of 
the  Korteweg- de Vries and the $1d$ Nonlinear Schr\"odinger equations (NLS), 
seen as solutions of the KP-I or the $2d$ NLS equations respectively
and subject to periodic transverse perturbations.
The linear instability in both cases was known. More precisely, in the KP-I case one has a complete 
understanding of the possible unstable modes for any fixed transverse frequency while in the NLS case
unstable modes where detected thanks to the Zakharov-Rubenchik bifurcation argument for small 
transverse frequencies. The possibility of describing all unstable modes in the KP-I case seems to be related to
the Lax pairs structure of the KP-I equation (sometimes called complete integrability). The  
Zakharov-Rubenchik bifurcation argument is a more general feature but does not seem to apply in some important
cases such as the gKP-I equation, a case which is in the scope of the applicability of the present paper.
Our goal here is to present a general transverse nonlinear 
instability theory of solitary waves,   assuming  the spectral  instability 
 of the solitary wave,   for Hamiltonian PDE's obeying to some structural assumptions
described below,  the main one being that, in some sense, the transverse
perturbation and the $1d$ dispersion operator should have the same sign.
More precisely, we state  two  instability results, one for transverse periodic boundary
 condition and one  where the transverse direction is unbounded
    and the  perturbations are   localized. This last case was not
     studied  in  our previous work \cite{RT}  and requires more
     work in the study of low frequencies.
We 
also present a criterion  to detect unstable modes,  and thus
 to prove linear instability,   inspired by the work of Groves-Haragus-Sun \cite{GHS}, 
which is different and more flexible than the one presented in our 
previous work \cite{RT} for NLS.   
 Finally, 
we check  that  our general theory
 can be applied to prove the  linear and nonlinear instability of 1d solitary waves in 
 the generalized KP-I equation, the 2d NLS equation,  a Boussinesq  type equation, 
  the Zakharov-Kuznetsov equation and the KP-BBM equation.
 

Our method mainly depends on the Hamiltonian structure of the equation and
we hope that the ideas of this paper may be extended to more general, not
necessarily linear transverse perturbations. In particular, we hope that our
approach may be useful to get transverse instability for some 
 more complicated 
 fluid mechanics models.

The  paper is organized as follows.
We  first describe the general framework and our assumptions.
Then  we state two abstract instability results under the
additional assumption of the existence of an unstable mode of the linearized
equation. 
Some of our assumptions will be easily verified in the applications. Other
assumptions such as the existence of  multipliers or the bounded frequencies resolvent estimates
are not a general feature in the considered framework. For that reason in the
later sections we present criteria insuring the validity of these assumptions
 and in particular, a criterion for the existence of unstable eigenmodes.
 These criteria will be usefull to analyze  our concrete examples. 
In the last section of the paper, we apply the general theory to various examples.
 \bigskip

{\bf Acknowledgement.}
We are indebted to Jean-Claude Saut for several discussions on the subject and
in particular for providing us the reference \cite{GHS}.

\section{General framework and  results}
\subsection{The unperturbed model}
\label{s1D}
For $s$ a real number, we consider the Sobolev spaces $H^s\equiv H^s(\R;\R^d)$, 
where $d\geq 1$ is an integer
 and we denote its norm by $|\cdot  |_{s}$. The $L^2$ norm will be simply denoted
  by $|\cdot|$ and the $L^2$ scalar product by $(\cdot, \cdot)$.  We 
consider the equation
\beq\label{ham1pak}
\partial_{t} u = J(L_{0}u+\nabla F(u)),
\eeq
where $F\in C^{\infty}(\R^d;\R)$, $F(0)=0$ and
 the linear operators $J$ and $L$ are such that : 
\begin{itemize}
\item $J$   is  a Fourier multiplier which is skew-symmetric  for the $L^2$
  scalar product  with domain containing $H^1$ (thus $J$ is
 of order at most one)  and  such that  $\mbox{Ker} J =
\{ 0\}$.  
\item   $L_{0}$  is a  Fourier multiplier which is a  symmetric operator  with a self
adjoint realisation on $L^2(\R;\R^d)$  with   domain $D(L_0)$ containing $ H^2$. 
  Moreover,  $L_0$ is coercive, 
   \beq
   \label{coerc}
C^{-1}|u|_{1}^2\leq (L_0u,u)\leq C |u|_{1}^2\,.
   \eeq
\end{itemize}
 Note that since $J$ and $L_{0}$ are Fourier multipliers, they commute with the
 $x$ derivative and hence we have that $J \in \mathcal{B}(H^s, H^{s-1})$
  and  $L_{0} \in \mathcal{B}(H^s, H^{s-2})$ for every $s$.

Equation (\ref{ham1pak}) can  thus be written in the Hamiltonian form
$$
\partial_{t}u=J\nabla H(u),\quad H(u)=\frac{1}{2}(L_0 u,u)+\int_{-\infty}^{\infty}F(u)\, 
 dx\,.
$$
One may imagine situations when $J$ and   $L_0$ are  of  higher orders. In these cases some modifications of the 
considered framework should be done. However, in all our examples $L_0$  is  of order $2$ and $J$ of order $0$ or $1$.

We are interested in the stability of 
stationary solutions of \eqref{ham1pak}. Since $J$ is into, they are critical
points of the Hamiltonian $H$, i.e. we have 
$\nabla H(Q) =  L_{0}Q+\nabla F(Q)=0$. We  focus on the case where $Q$ is
smooth,  $Q\in H^{\infty}$. Next, we consider the linear operator  associated
to the second variation of the Hamiltonian at $Q$ : 
$$  
L\equiv D_{u}(\nabla H)(Q) = L_{0}+ R, \quad Ru= \nabla^2 F(Q) u.
$$
Note  that $R$ is a bounded operator on  $H^s$ for every $s \geq 0$ since $Q$ and $F$ are smooth.
Consequently, $L$ is a self adjoint operator on $L^2$ with domain $D(L_0)$.  
Our main assumption on $L$ is that its spectrum   is under the form 
\begin{equation}
\label{spectre}
\sigma(L) = \{\mu\}\cup \{0\} \cup \Sigma, 
\end{equation}
where $\mu<0$ is a simple  eigenvalue, $0$ is an eigenvalue of finite multiplicity
and $\Sigma \subset [\alpha, + \infty)$ for some $\alpha >0$. Moreover,  the
eigenspaces associated to $\mu$ and zero are made of smooth eigenvectors (i.e
which are in  $H^\infty)$. 
Many of our arguments remain valid if $\sigma(L)\cap ]-\infty,0]$ contains a
finite number of eigenvalues of finite multiplicities.
We will be interested in situations where $Q$ is a stable object for
  \eqref{ham1pak}. Note that the spectral assumption \eqref{spectre}
   is one of the main assumption which allows to prove the
    stability of $Q$ by the Grillakis-Shatah-Strauss method \cite{GSS}.
\subsection{The transversally  perturbed model}
We are interested in the stability of $Q$ when \eqref{ham1pak} can be embedded in a larger Hamiltonian equation
\beq\label{ham2}
\partial_{t} u = \mathcal{J}(\partial_{y}) (L_0u+ \nabla F (u) +\mathcal{S}(\partial_{y}) u ),
\eeq
where now $u$ also  depends on $y$  with $y \in \mathbb{T}_{a}=$$\mathbb{R}/
2\pi a  \mathbb{Z}$ or $y \in \mathbb{R}$ and  $L_0$ acts in a natural
way on functions of $2$ variables.
The operators $\mathcal{J}(\partial_{y}), $   $\mathcal{S}(\partial_{y})$ are  operator valued
Fourier multipliers in $y$, i.e. if $\mathcal{F}_{y}$  stands for the Fourier transform in $y$, we have 
$$ \mathcal{F}_{y} (\mathcal{S}(\partial_{y} )  u )(k) = S(ik) 
\mathcal{F}_{y}(u)(k) ,  \quad  \mathcal{F}_{y} (\mathcal{J}(\partial_{y} )  u )(k) = J(ik) 
\mathcal{F}_{y}(u)(k).$$
Moreover, $S(ik)$ and $J(ik)$ are now Fourier multipliers in $x$.
In the following, we still denote by $(\cdot, \cdot)$ and $|\cdot |_{s}$ the complex scalar product of 
$L^2(\mathbb{R}, \mathbb{C}^d)$ and the $H^s$ norm for complex valued functions respectively.
   
\subsubsection{Assumptions on the operator $J(ik)$}
\label{sJ}
For every $k$, $J(ik)$ is a Fourier multiplier
 such that: 
 \begin{itemize}
 \item $J(ik)$  and $J(ik)L_{0}$   are     skew symmetric  on $L^2(\R)$, $J(0)=J$,  
 \item  the domain of $J(ik)$  contains  $H^1$, $\mbox{Ker }J(ik)= \{0\}$ 
  and we have  the uniform bound
 \beq
 \label{Jik}
  \exists\, C>0, \, \forall\, k, \quad    |J(ik) u |  \leq C |u|_{1}, \, \forall u \in H^1,
  \eeq 
 \item   The commutator $[R, J(ik)]$ is a uniformly bounded operator
  on $L^2$ :   
  \beq
  \label{JR}
 \exists\, C>0, \quad    \forall\, k, \quad   \big|\big( [R, J(ik)] w, w \big) \big| \leq C |w|^2.
   \eeq
 \end{itemize} 
 Note that since $J(0)=J$,   $\mathcal{J}(\partial_y) u = Ju$ if $u$ depends only on $x$.
   We also point out that the assumption \eqref{JR} is obviously verified
    when $J$ is a bounded operator on $L^2$.
 
\subsubsection{Assumptions on the operator $S(ik)$} 
\label{sS}   
For every $k$,  $S(ik)$  is a  Fourier multiplier
 such that: 
\begin{itemize}
\item  $S(ik)$ is  non-negative and symmetric, $J(ik)S(ik)$ is skew symmetric, $S(0)=0$,
\item $S(ik)$ has  a self-adjoint realisation on   $L^2$  with
domain $\mathcal{D}_{S}$ independent of $k$ for $k \neq 0$,
\item   $J(ik)S(ik)J(ik)$  and   $J(ik) S(ik) \partial_{x} $  belong to $\mathcal{B}(H^2(\R), L^2(\R))$, 
 \item Let us set  $ |w|_{S(ik)}^2 \equiv ( w, S(ik) w)$, then  there exists a
non-negative continuous function (possibly unbounded) $C(k)$  such that
\beq\label{JSw}
|J(ik) S(ik) u|_{L^2} \leq C(k) |u|_{S(ik)}, \quad \forall\,\, u \in \mathcal{D}_{S}\,.
\eeq
\end{itemize}
Note that \eqref{ham2} also has an Hamiltonian structure
 with Hamiltonian given by 
$$ \mathcal{H}(u) = \int \Big(\frac{1}{2} (L_{0}u, u ) +  
\frac{1}{2}(\mathcal{S}(\partial_{y})u, u ) + \int_{\mathbb{R}} F(u)\, dx \Big) \, dy.$$
Moreover, we also point out that  $Q$ is still a stationary solution of \eqref{ham2}
 and more generally that  if $u$ is a (reasonable) solution of \eqref{ham2}
  which does not depend on $y$, then $u$ actually solves \eqref{ham1pak}.
  
\subsubsection{Compatibility between $S(ik)$ and $L$}
 We assume that 
 there exists $K$  and $c_{0}>0$ such that for every $|k| \geq K$,
\beq\label{klarge}
(Lv, v ) + (S(ik)v, v ) \geq c_{0}  |v |_{1}^2, \quad \forall\,\, v \in
H^2 \cap \mathcal{D}_{S}.
\eeq
This is one of our main  structural assumption which  roughly says
 that  $S$ and $L_{0}$ have the same sign. This assumption is valid
  for example for the KP-I equation and the 2d NLS equation but not
   for the KP-II equation or the hyperbolic Schr\"odinger equation.

\subsection{The resolvent equation}
 
In this subsection, we  state  our assumptions on  the linearization of
(\ref{ham2}) about $Q$.  
Since $\mathcal{S}(\partial_{y})$ is a linear map, 
the linearization of (\ref{ham2}) about $Q$ reads
\begin{equation}\label{linearizatz}
v_{t}=\mathcal{J}(\partial_{y}) (L+\mathcal{S}(\partial_{y}))v.
\end{equation}
\begin{defi}
An unstable mode for (\ref{ham2}) is a function $U\in L^2 \cap \mathcal{D}(S(ik))$ such that for some
$\sigma\in \C$ with  $\Re(\sigma)> 0$ and some $k\in \R$,   the problem
(\ref{linearizatz}) has a solution of the form
\beq
\label{imode}
v(t,x,y)=e^{\sigma t}U(x)e^{iky}\,.
\eeq
We call $\sigma$ the amplification parameter and $k$ the transverse frequency associated to $U$.
\end{defi}
Thus if $U$ is an unstable mode then it is a solution of the eigenvalue  problem
\beq\label{spectral}
\sigma  U = J(ik) (L+S(ik))U,\quad U\in L^2(\R;\C^d).
\eeq

\subsubsection{Assumption of existence of an Evans function and 1d stability}
\label{asE}
We assume that there exists a function $D(\sigma, k)$ (Evans function) such that  for every $k$, $D(\cdot, k)$
   is analytic in $\mbox{Re } \sigma >0$ and such that
    there exists an unstable mode \eqref{imode} if and only if
     $D(\sigma, k)=0$.  We also  assume that all the possible  unstable eigenmodes are smooth
     ($H^\infty$)
     and that $Q$ is spectrally stable with respect to  one-dimensional  perturbations  which  reads: 
   \beq
     \label{evans1D}
     D(\sigma, 0) \neq 0, \quad  \mbox{Re }\sigma>0.
     \eeq

     A concrete criterion for the existence
      of the Evans function will be given in section \ref{criteria}.
      In most examples we have in mind,  \eqref{spectral} can be reduced
       to an ordinary differential equation and hence, 
       as usual,  the Evans
      function will be defined as a Wronskian determinant associated to an ODE
      obtained after some manipulations from (\ref{spectral}).

\bigskip
Next,  let us  consider the  resolvent equation for $Re(\sigma)>  0$ 
\begin{multline}\label{res2} 
\sigma  U = J(ik) LU +  J(ik) S(ik)U + J(ik) F,\\ 
U\in H^\infty(\R;\C^d) \cap \mathcal{D}(S(ik), \, F \in H^\infty(\mathbb{R}, \mathbb{C}^d )\,.
\end{multline}
\subsubsection{Bounded frequencies  resolvent bounds in the periodic case}
   \label{hypper}
    We assume that there exists $q$ such that
     for every $k$ and every  $\mathcal{K}$ compact set in $\mbox{Re
     }\sigma>0$, every $s\geq 0$, there exists $C_{k, \mathcal{K}, s}$ such that if
       $D(\cdot, k)$ does not vanish on $\mathcal{K}$, then  for every $F \in H^\infty(\R)$, there is 
         a unique solution  $U \in H^\infty(\R)  \cap \mathcal{D}(S(ik))$ of \eqref{res2}
           which satisfies
   \beq
   \label{resper}
    |u|_{s} \leq C_{k, \mathcal{K}, s} | F |_{s+q}, \quad \forall \sigma \in \mathcal{K}.
    \eeq
 
    \subsubsection{  Bounded frequencies   resolvent bounds in the  localized case}
    \label{hyploc}
   When $k$ is a continuous variable,  we need   some uniform dependence
    in $k$ in the regime $k\sim 0$. We shall assume that the Evans function  $D$  is analytic
     in $(\sigma, k)$ for $\mbox{Re }\sigma >0$ and $k \neq 0$ and
      that there exists an analytic continuation $\tilde{D}(\sigma, k)$
       which is analytic in $\{\mbox{Re }\sigma >0\} \times \mathbb{R}$.
       Moreover,   we  assume a strong $1D$ stability
   \beq
     \label{evans1Dt}
    \tilde{D}(\sigma, 0) \neq 0,   \quad 
    \forall \sigma, \, \mbox{Re }\sigma>0
     \eeq 
     and the uniform (also with respect to $k$) resolvent bound : there exists $q\geq 0$ such that  
  for every compact set $\mathcal{K}$ in $\{\mbox{Re }\sigma >0\}$
   and $M>0$,  there exists $C_{\mathcal{K},M, s}$ such that
 if 
  $\tilde{D}(\sigma, k)$ does not vanish on $\mathcal{K}\times (0, M] $, then
    for every $F \in H^\infty$,   there is 
         a unique solution  $U \in H^\infty  \cap \mathcal{D}(S(ik))$ of \eqref{res2}
          which satisfies
   \beq
   \label{resloc}
    |u|_{s} \leq C_{\mathcal{K}, M,  s} | F |_{s+q}, \quad
     \forall (\sigma, k)\in \mathcal{K}\times (0, M].
    \eeq

As we shall see below, in most examples the existence of the  Evans function
     and the bounds \eqref{resper}, \eqref{resloc} can be obtained
      by ODE techniques. We shall give below a simple criterion
         which allows to obtain \eqref{resper}, \eqref{resloc}.
         We also point out that  we allow the case where $\tilde{D}(\sigma, 0)$
          is different from $D(\sigma, 0)$ since we have not assumed
         continuity of $D$ at $k=0$. Typically $D(\sigma,k)$ is the
         determinant of a matrix of fixed size for $k\neq 0$ and $D(\sigma,0)$
         is the determinant of a smaller matrix. 

\bigskip
 As we shall prove, the assumptions \eqref{spectre}, \eqref{klarge} and the structural properties
  of the operators given in sections \ref{sJ}, \ref{sS} are sufficient
   to ensure nice resolvent bounds in the energy norm $H^1$  for large $|\mbox{Im }\sigma|$.
   The following assumption will be used to  get  the estimates of
    higher order derivatives.
\subsubsection{Existence of a multiplier}\label{Ms}
We suppose that for every $s\geq 2$, there exists a self-adjoint operator $M_{s}$ such that
there exists $C>0$ with 
\begin{equation}\label{Ms1}
 | (M_{s}u, v ) | \leq  C |u|_{s}\, |v|_{s},\quad
(M_{s}u, u ) \geq |u|_{s}^2 - C |u|_{s-1}^2
\end{equation}
and
\begin{equation}\label{Ms2}
\Re(J(ik)(L+ S(ik))u, M_{s} u ) \leq C_k  |u|_{s}\, |u|_{s-1} .
\end{equation}
The assumption (\ref{Ms2}) will play a key role for the control on higher
derivatives in a resolvent analysis below.
In the cases of "semi-linear"
problems we will be able simply to choose $ M_s=\partial_{x}^{s-1}L\partial_{x}^{s-1}$.

\subsection{The nonlinear problem}
\label{asnon}    
Finally, we make a set of  assumptions on  the nonlinear problem \eqref{ham2}.  
Denote by $\mathbb{H}^s$ the Sobolev type spaces on $\R\times \R$ or $\R\times
\T_{a}$  with the  norms $\|\cdot\|_{s}$. 
We denote by $\|\cdot\|$ the norm of $L^2$. Consider the problem
\begin{equation}\label{ham222}
\partial_{t} u = \mathcal{J}(\partial_y) ( L_{0}u+\nabla F(u^a+u)-\nabla F(u^a)
+\mathcal{S}(\partial_{y}) u )+\mathcal{J}(\partial_y) G,\,\,
u(0)=0,
\end{equation}
where $u^a$ is a smooth function bounded with all its derivatives
and $G\in C(\R;{\mathbb H}^s)$ for every $s$.
We suppose that the problem \eqref{ham222} is locally well-posed in the sense
that for every $u^a$ and $G$ satisfying the previous assumptions there exists
a time $T>0$ and a solution of (\ref{ham222}) in $C([0,T];{\mathbb H}^s)$
 for every $s \geq s_{0}$ ($s_0>0$ being sufficiently large), unique in a suitable class.
 Finally,  we  assume the that  tame estimate \beq
 \label{tame}
  \Big|\Big(\!\!\Big( \partial_{x}^\alpha \partial_{y}^\beta
  \mathcal{J}(\partial_{y})
\big( D\nabla F( w+v)\cdot  v\big), 
   \partial_{x}^\alpha \partial_{y}^\beta v
  \Big)\!\!\Big) \Big| \leq \omega\big(   \|w \|_{W^{s+1, \infty} }+ \| v\|_{s} \big) \| v\|_{s}^2
  \eeq
  holds for every $\alpha$, $\beta$, 
  $\alpha+\beta=s$, 
where $\omega$ is a continuous non-decreasing function with $\omega(0)=0$
 and $(\!(\cdot, \cdot )\!)$ is the $L^2$ scalar product for functions
  of two variables.
  
  This last assumption together with the properties
   of the operators $\mathcal{J}$ and $L_{0}$  will ensure the existence of an $\mathbb{H}^s$
    energy estimate for \eqref{ham222}.

\subsection{Statement of the abstract results}
Let us state our first instability result for (\ref{ham2}) with $\R\times \T_{a}$ as a spatial domain.
\begin{theo}[Nonlinear transverse periodic instability]\label{theoper} 
Consider the Hamiltonian equation (\ref{ham2}) and suppose that the assumptions of the previous sections 
hold true, except the assumptions of Sections~\ref{hyploc}. 
Assume also that there exists an unstable mode with corresponding transverse frequency $k_{0} \neq 0$.
Then we have nonlinear instability  of \eqref{ham2} defined on $\mathbb{R}\times \mathbb{T}_{2\pi/k_{0}}$.
More precisely for every $s\geq 0$, there exists $\eta>0$ such that for every $\delta >0$, there exists 
$u_{0}^\delta \in \mathbb{H}^\infty(\mathbb{R}\times \mathbb{T}_{2\pi/k_{0}})$ 
and a time $T^\delta \sim |\log \delta | $  such that 
$
\|u_{0}^\delta - Q \|_{s} \leq \delta 
$
and the solution $u^{\delta}$ of \eqref{ham2} with data $u_{0}^\delta$ remains 
 in $ \mathbb{H}^s$ on $[0, T^\delta]$ and satisfies
$ 
d( u^\delta(T^\delta),  \mathcal{F})  \geq  \eta  
$
where  $\mathcal{F}$  is the space of $L^2(\R)$ functions depending only on $x$ and
$
d( u,  \mathcal{F}) = \inf_{v \in \mathcal{F}} \|u - v \|. 
$
\end{theo}
Notice that we have a strong instability statement since we measure the
initial perturbation in a strong norm such as $\|\cdot\|_s$ while the
instability occurs in the weaker norm $L^2$. 
Our second result concerns fully localized perturbations.       
\begin{theo}[Nonlinear transverse localized  instability]\label{theoloc}
Consider the Hamiltonian equation (\ref{ham2}) and suppose that the assumptions of the previous sections 
hold true. Assume also that there exists an unstable mode with $k\neq 0$.
Then we have nonlinear instability  of \eqref{ham2} posed on $\mathbb{R}^2$.
More precisely for every $s\geq 0$, there exists $\eta>0$ 
such that for every $\delta >0$, there exists $u_{0}^\delta$ and a time $T^\delta \sim |\log \delta |$ such that 
$
\|u_{0}^\delta - Q \|_{s} \leq \delta 
$
and the solution $u^{\delta}$ of \eqref{ham2} with data $u_{0}^\delta$
remains defined on 
$[0,T^\delta]$, i.e. $u^\delta - Q \in \mathbb{H}^s, $ $\forall t \in [0, T^\delta]$
 and satisfies
$ 
d( u^\delta(T^\delta),  \mathcal{F})  \geq  \eta  
$
where again $\mathcal{F}$  is the space of $L^2(\R)$ functions depending only on $x$ and
$ 
d( u,  \mathcal{F}) = \inf_{v \in \mathcal{F}} \|u - v \|. 
$
\end{theo}
These theorems state that the existence of an unstable eigenmode implies
nonlinear orbital  instability of the solitary wave. Indeed, the orbit of $Q$  
under the action of all  the possible groups of   invariance of
\eqref{ham1pak} remain in $\mathcal{F}$. 
In particular our results exclude the possibility of orbital stability of $Q$ with respect to the
spatial translations. More precisely our result implies that
$$
\inf_{a\in\R}\|u(T^\delta)-Q(\cdot-a)\|\geq \eta.
$$

 There are many assumptions in these theorems, nevertheless,
  some of them will be very easy  to check on examples, for example
   the structural assumption \ref{sJ}, \ref{sS}. The ones which
    are more difficult to check are the assumptions
     \ref{asE}, \ref{hypper}, \ref{hyploc}, \ref{Ms} that is to say, the existence of
      an Evans function and of multipliers,  the  bounded frequencies resolvent bounds, and also the assumption
       on the existence of an unstable eigenmode. Consequently,
  the next sections are devoted to the  proof of more concrete criteria 
     which ensure that these assumptions are verified and which are easy
      to test on examples.

  Let us explain the main steps in the proof of Theorem 2.
   The  inspiration comes from the work of Grenier \cite{Grenier}
     in fluid mechanics problems. 
   We believe that this scheme is quite general and may be useful in other contexts.
\begin{enumerate}   
\item[1.] The first step is   to prove that  the possible  unstable modes
  in the sense of Definition~2.1 above necessarily belong to a compact set both with respect to the transverse frequency and the amplification parameter.
 This allows  to  find  the most unstable mode  i.e. with the  largest real part of the amplification parameter
 (note that there exists at least an unstable eigenmode by assumption) and to define a first  approximate  growing solution by a wave packet construction
  in the framework of Theorem \ref{theoloc}.
\item[2.]  The second   step is to evaluate, both from above and below, in a suitable norm (here it is $L^2$) the first approximate  growing solution given by step 1.  In the proof
 of Theorem \ref{theoloc}, we need to 
 use the  Laplace method   and  some properties of the curve $k\rightarrow \sigma(k)$.
\item[3.] The third  step is, following Grenier \cite{Grenier},  the construction of a refined approximate solutions which is carefully  estimated  from above.  Since we deal with Hamiltonian PDE's this step requires a different argument compared to  similar estimates for diffusive problems. Here we reduce the matters to resolvent bounds for $\sigma-J(ik)(L+S(ik))$ for $\sigma$'s with real parts larger that the amplification parameter of the most unstable mode and any $k$ in the (compact) set of possible transverse frequencies.
\item[4.] The last step is to estimate the difference between the refined approximate solution and the true solution on the interval $[0,T^\delta]$ by energy estimates. The analysis in this step is quite flexible and seems to apply each time we have $H^s$ energy estimates for the full $2d$ problem.

\end{enumerate}

  The  paper is organized as follows. In section \ref{instab}  we give
   a criterion for the existence of an unstable eigenmode,   in section \ref{criteria},
    we give criteria for the existence of the Evans function and the bounded
     frequencies resolvent bounds and in section \ref{multip}, we give a criterion
      for the existence of  multipliers satisfying (\ref{Ms1}), (\ref{Ms2}).  The two next sections
       are devoted to the proof of Theorem \ref{theoper} and \ref{theoloc}
        and finally, the last  section is devoted to the study of various examples for which
         we check that the general theory can be applied.

\section{A sufficient condition for the existence of an unstable mode} 
\label{instab}
In this section, we give a simple criterion which ensures the existence of an
unstable eigenmode. This criterion is inspired by the work \cite{GHS}.
Consider the  symmetric operator  defined by
$$
M_{k} = J (ik)LJ(ik)  + J (ik) S(ik) J(ik).
$$
Since $J(ik)S(ik)J(ik) \in \mathcal{B}(H^2, L^2)$  by assumption \ref{sJ}, we get
 that the domain $\mathcal{D}$ of $M_{k}$  contains $H^4$
 (indeed, $J(ik)$ is at most a  first order operator and $L$ is
  a second order operator).
   
A simple criterion for the existence of an unstable eigenmode is given by the following statement.
\begin{lem}\label{eigen} 
Assume that for every $k$ and every $u$ real-valued  $J(ik)u$ and $S(ik)u$
are also real valued. Next, 
assume that  there exists $k_{0}\neq 0$ such that zero is a simple
eigenvalue of $M_{k_{0}}$ with corresponding real-valued 
 nontrivial eigenvalue $\varphi\in
H^{\infty}$ normalized so that $\|\varphi \|_{L^2(\R)}=1$.
Finally,  assume that $M_{k_{0}}$  is a Fredholm map of 
index zero,  that $M_{k}$ depends smoothly on $k$ for $k$
 close to $k_{0}$ and   the non degeneracy condition 
\beq\label{hyplem}
\Big([ \frac{d} {dk }M_{k}]_{k=k_0} (\varphi), \varphi \Big) \neq 0\,.
\eeq
Then there exists $k$ in a vicinity of $k_{0}$ and  $\sigma>0$  such that 
there exists an unstable mode with amplification parameter $\sigma$ 
and transverse frequency $k$.
\end{lem}
As we shall see, this criterion can be  used on many examples.  
\begin{proof}[Proof of Lemma~\ref{eigen}.]
We need to solve the problem
$$
\sigma v=J(ik)Lv+J(ik) S(ik)v,\quad v\in L^2(\R)
$$   
for $k$ close to $k_0$ and $\sigma$ close to $0$. 
 We shall seek for $\sigma$ real and $v$ real-valued.
  This is legitimate since by assumption $J(ik)v$ and $S(ik)v$
   are real-valued if $v$ is real-valued.
We shall look for 
$k=k(\sigma)$ with $k(0)=k_0$.  Since $\mbox{Ker }J(ik)=\{0\}, $ we look for $v$ under the form
$v=J(ik)u$, $u\in L^2$.
Therefore, we need to solve the problem $F(u,k,\sigma)=0$, where
$$
F(u,k,\sigma)=M_{k}(u)-\sigma J(ik)u\,.
$$
We search for $u$ of the form $u=\varphi+w$ with $w\in\tilde{{\mathcal D}}\equiv\{u\in 
\mathcal{D}\cap L^2(\mathbb{R}, \mathbb{R}^d): 
(u,\varphi)=0\}$. Define
$$
G(w,k,\sigma)\equiv
F(\varphi+w,k,\sigma)=M_{k}\varphi-\sigma J(ik)\varphi+M_{k}w-\sigma J(ik)w
$$
as a map on $\tilde{{\mathcal D}}\times \mathbb{R}\times\mathbb{R}$ to $L^2$. Note that we have
 $$ G(0, k_{0}, 0)=M_{k_{0}} \varphi = 0$$
  since $\varphi$ is an eigenvector of $M_{k_{0}}$ by assumption.
Next for $(w,\mu)\in \tilde{{\mathcal
    D}}\times\mathbb{R}$, we have 
$$
D_{w,k}G(0,k_0,0)[w,\mu]=M_{k_0}w+\mu\big([ {\frac{d}{dk} }M_{k}]_{k=k_0}\, \varphi\big)\,.
$$
Thanks to (\ref{hyplem}) the  linear map $D_{w,k}G(0,k_0,0)$ is a
bijection from $\tilde{{\mathcal D}}\times \mathbb{R}$ to 
$L^2(\mathbb{R}, \mathbb{R}^d)$.
Consequently, by  the implicit function theorem,  for $\sigma$ close to $0$ there exist $w(\sigma)\in\tilde{{\mathcal
    D}}$ and
$k(\sigma)\in\mathbb{R}$ with $w(0)=0$ and $k(0)=k_0$ such that
$
G(w(\sigma),k(\sigma),\sigma)=0.
$
This completes the proof of Lemma~\ref{eigen}.
\end{proof}
\section{Criteria for the existence  of the Evans function and the bounded frequencies resolvent bounds }
\label{criteria}
In this section we describe some concrete criteria in order to ensure
the assumptions of sections \ref{asE}, \ref{hypper} and \ref{hyploc}.
The first assumption roughly says that we can reduce the eigenvalue problem
(\ref{spectral}) to an ordinary differential equation.
\subsection{Reduction to an ODE}
\label{reduc}
We thus assume that there exists a  Fourier multiplier  $R(\sigma, k)$
  such that $R(\sigma, k)  \in \mathcal{B}(H^{s+l_{k}}, H^s))$ for every $s\geq0$
 and  that $\mbox{Ker } R = \{0\} $.
     Moreover, we assume the block structure 
  \beq
  \label{factorisation}
   \sigma R (\sigma, k)-  R(\sigma, k)  J(ik) \big( L+ S(ik) \big) = \left( \begin{array}{cc} P_{1}(\sigma, k) & 0
    \\ P_{2}(\sigma, k) & E(\sigma, k)  \end{array} \right)
   \eeq
   where :
   \begin{itemize}
   \item For every $k$,  $P_{1}(\sigma, k)$ is a $r\times r$ matrix of differential operators of order $m_{k}
    \geq 1 $  with coefficients
    which depend analytically on $ \sigma $,
    \beq
    \label{P1}
    P_{1}(\sigma, k) = \partial_{x}^{m_{k}}{\rm Id} + \cdots,
    \eeq
    \item for every $k$,  $P_{2}(\sigma,k)$ is an operator of order $\leq m_{k}-1$ i.e.
             $P_{2}(\sigma, k ) \in \mathcal{B}(H^{s+m_{k} - 1 }, H^s)$ for every $s \geq0$
      \item  For every $k$,  $E(\sigma, k)$ is invertible  and $E(\sigma, k)^{-1}
       \in \mathcal{B}(H^s, H^s)$ for every $s\geq 0$.
      \item There exists $(l,m)$ such that   for every $k \neq 0$,  $(l_{k}, m_{k})= (l,m)$  
       and $l_{0} \leq l$, $m_{0} \leq m.$
  \end{itemize}
  Moreover, all the operators depend continuously on $\sigma$ for $\mbox{Re }
  \sigma>0$  for each fixed $k$.
 \bigskip
  
  Because of the triangular block structure \eqref{factorisation}, 
   the study of the resolvent equation \eqref{res2} can be reduced to
     the study of the ordinary differential equation
   \beq
   \label{P2eq} P_{1}(\sigma, k) u_{1}=  (R(\sigma,k)J(ik)F)_{1}
   \eeq
   by using the block decomposition  $U=(u_{1}, u_{2})^t \in \mathbb{C}^{r} \times \mathbb{C}^{d-r}$.
    Note that we allow the possibility that  $r=d$,  which means that
     the resolvent equation can be directly  reduced to an ordinary differential equation
     by applying the operator $R(\sigma, k)$. 
 
   We  can rewrite \eqref{P2eq} as a  first order ordinary differential equation
\beq\label{edo}
\frac{ d V}{dx}  = A(x,\sigma, k) V + \mathbb{F},
\eeq
where $A(x,\sigma, k)\in \mathcal{M}_{N_{k}}(\C)$, $ N_{k}=m_{k}\,r$ is a
matrix which depends smoothly on $x$,
analytically on $\sigma$
 and 
\begin{equation}
\label{ani}
\mathbb{F}=(0,\cdots,0, (R(\sigma,k)J(ik)F)_{1}). 
 \end{equation}
 Note that $A(x,\sigma, k)$ is  in general  not ``continuous'' at $k=0$, since
   for $k=0$, the dimension of the matrix may be different.
   
 With our reduction  assumptions, we  have unstable eigenmodes 
 if and only if the ODE \eqref{edo}
  with $\mathbb{F}=0$ has a nontrivial  $L^2$ solution.

\subsection{Asymptotic behavior and consistent splitting}
\label{cons}
  We add the assumption
 that there exist $A_{\infty}(\sigma, k)$ and $C>0$, 
$\alpha >0$, such that for every $x,k\in\R$,  and every $\sigma$, 
\beq
\label{asympt} 
|A(x, \sigma, k ) - A_{\infty}(\sigma, k)|
\leq C e^{ -\alpha |x |}, 
\eeq
and   that   the spectrum of $A_{\infty}(\sigma, k)$
{\bf does not} meet the imaginary axis for $Re(\sigma)> 0$.
 
 \bigskip 
\subsection{Existence of the Evans function}
\begin{lem}
\label{existevans}
Under the assumptions of sections \ref{reduc}, \ref{cons}, 
 there exists a function $D(\sigma, k)$ (Evans function)
 which is analytic  in
$Re(\sigma)> 0$, for every  $k $  and such that $D(\sigma, k)=0$ if and only if there exists a non trivial  eigenmode solution   of \eqref{spectral}.
\end{lem}
\subsubsection*{Proof}
By classical arguments (see e.g. \cite{AGJ}),  the  assumptions of section \ref{cons} 
allows to define an Evans function $D(\sigma, k)$ for \eqref{edo}  which is an analytic function in
$Re(\sigma)> 0$, for every  $k  $  and such that $D(\sigma, k)=0$ if and only if there exists a non trivial 
$L^2(\R;\R^{m_{k}\,r})$ solution of  $V'= A(x,\sigma, k) V$  which is actually exponentially decreasing. Thanks to the reduction  assumptions \ref{reduc} above, 
 this  is equivalent to the existence of  a nontrivial solution of \eqref{spectral}.

\subsection{Resovent estimates in the periodic case}
Under the above assumptions, we can prove : 
\begin{lem}
\label{lemper}
Let $R(\sigma, k)$ satisfying assumptions \eqref{reduc}
 and \eqref{cons}, then, there exists $q \geq 0$
 such that   for every
  $k$, every $s\geq 0$ and every compact $\mathcal{K} \subset \{\mbox{Re } \sigma >0\}$, 
 there exists $C_{k, \mathcal{K}, s}$ such that if
       $D(\cdot, k)$ does not vanish on $\mathcal{K}$, then there is 
         a unique solution  $U \in H^\infty  \cap \mathcal{D}_{S}$ of \eqref{res2}
          for every $F \in H^\infty$ which satisfies
   \beq
   \label{resperlem}
    |u|_{s} \leq C_{k, \mathcal{K}, s} | F |_{s+q}.
    \eeq
\end{lem}
In other words, if one can prove the existence of $R(\sigma, k)$
then one  get the resolvent bounds \eqref{resper}  on every compact
 which does not contain unstable eigenmode.
 
 \subsection{Resolvent estimates in the localized case}
 \label{reshyploc}
  To get \eqref{resper} in the localized case,
   we need  some assumptions on the dependence of the various
    objects with respect to  $k$.
      We   assume that:
 \begin{enumerate}
 \item[i)]  $R(\sigma, k)$, $P_{2}(\sigma, k),$  $E(\sigma, k)$
    depend continuously on $ (\sigma, k)$ for $k \neq 0$, $\mbox{Re }\sigma>0$ and
    have continuous extensions up to $\{\mbox{Re }\sigma>0\} \times \mathbb{R}$.
 \item[ii)]    $P_{1}$ and thus $A$ and $A_{\infty}$ 
     are  analytic for $k \neq 0$, $\mbox{Re }\sigma>0$ and have
      analytic extensions up to   $\{\mbox{Re }\sigma>0\} \times \mathbb{R}$.
 \end{enumerate}      
 Next, since  the spectrum of $A_\infty(\sigma, k)$
does not meet the imaginary axis for $\Re(\sigma)> 0$ and  $k \neq 0$, we can define a projection on the 
stable subspace  of $A_{\infty}$ which is analytic in $\sigma$ and $k$  by the Dunford integral
$$ 
P_{\infty}(\sigma, k ) = \int_{\Gamma} (z - A_{\infty}(\sigma, k) )^{-1} \, dz,
$$
where $\Gamma$ is a contour  which encloses all the negative real part eigenvalues of  $A_{\infty}$, 
the projection on the unstable subspace is then given by $Id -P_{\infty}$.
 Note that we had  assumed  that  $A_{\infty}(\sigma, 0+)=\lim_{k\rightarrow 0}A_{\infty}(\sigma, k)$ exist
  but we  allow the presence of  eigenvalues on the imaginary axis.            
We  nevertheless  assume:
\begin{enumerate}
\item[iii)]  the projection 
$P_{\infty}(\sigma, k)$ can be continued 
analytically
to  $\{Re(\sigma)> 0\} \times \mathbb{R}$. 
\end{enumerate}
 This implies  thanks to the Gap Lemma  (\cite{GZ}), \cite{KS})  that the Evans function can also be continued
  analytically  to  $\{Re(\sigma)> 0\} \times \mathbb{R}$. The continuation
   of the function will be denoted by $\tilde{D}(\sigma, k)$.
   Recall that $\tilde{D}(\sigma,0)$  may  be  different from $D(\sigma, 0)$.
Indeed $A(x,\sigma, k)$ is not continuous  at zero and
 hence $A(x, \sigma,0) \neq \lim_{k\rightarrow 0 } A(x, \sigma, k)$. By construction,
  the same  difference holds for the Evans function.
  
   Finally, we  also assume : 
\begin{enumerate}
\item[iv)]  for every compact  set  $K$ of
$\{Re\, \sigma >0 \}$, and every $s \geq 0$, there exists $C>0$ such that for every eigenvalue
$\mu(\sigma, k) $ of $ A_{\infty}(\sigma, k)$ 
\begin{eqnarray}
\label{compensation}
 & & \|R(\sigma, k)  - R(\sigma, 0^+) \|_{\mathcal{B}(H^{s+l}, H^s)}
 +  \|J(ik)  - J(0)  \|_{\mathcal{B}(H^{s+1}, H^s)}  \\
\nonumber  & & + \|
 R(\sigma, k) J(ik) S(ik) \|_{\mathcal{B}(H^{s+m},  H^s ) } \leq C \rho(k, K)
\end{eqnarray} 
where
$$ \rho(k, K) = \inf_{\sigma \in K, \mu(\sigma, k) \in Sp A_{\infty}(\sigma, k)}  |\Re\,\mu(\sigma, k)|$$ 
for every $k$ in a small disk $D(0,r) \backslash \{0\}$ and $\sigma \in K$,
where $l$,  $m$ and $k$ are defined in section~\ref{reduc}.
\end{enumerate}
Note that since $S(0)=0$,   this assumption is nontrivial only when
 there exists an eigenvalue of $A_{\infty}(\sigma, k$
  such that  $\Re\, \mu(\sigma, k)$ vanishes at $k=0$.

Then, we can prove the following statement.
\begin{lem}\label{lemloc}
  Assuming the existence of $R(\sigma, k)$ given
   by assumptions \ref{reduc}, \ref{cons} and assumptions
    i)-iv) above,  
   then  there exists $q \geq 0$ such that 
  for every $s\geq 0$, every compact $\mathcal{K}
   \subset \{\mbox{Re }\sigma >0 \}$
    and $M>0$,   there exists $C_{\mathcal{K},M, s}$ such that if
       $\tilde{D}$ does not vanish on $\mathcal{K}\times [0, M ]$
        and $D$ does not vanish on $\mathcal{K}$, then, 
     for every $F \in H^\infty$,     there is 
         a unique solution  $U \in H^\infty  \cap \mathcal{D}(S(ik))$ of \eqref{res2}
          which satisfies
   \beq
   \label{resloclem}
    |u|_{s} \leq C_{ \mathcal{K},M, s} | F |_{s+q}, \quad  \forall \sigma \in \mathcal{K},
     \, \forall k \in (0, M].
    \eeq
 
 \end{lem}
 Consequently, we have given criteria which allow  to obtain \eqref{resloc}

 \subsection{Proof of Lemma \ref{lemper}} 
   
   By  using $R(\sigma, k)$ and setting $w=(u_{1}, u_{2})^t$, we can rewrite 
$$
 \sigma   w  = J(ik)(L w + S(ij) w ) + J(ik) F  
$$
as
\beq\label{abst}
V_{x} = A(\sigma, k , x) V + \mathbb{H}, 
\eeq
and 
$$   u_{2} = -  E(\sigma , k)^{-1} P_{2}(\sigma, k) u_{1} + E(\sigma, k)^{-1}\big
 (R(\sigma, k) J(ik)F)_{2}$$
with $V(x)=(u_{1},\cdots,\partial_{x}^{m_{k}- 1}u_{1}(x))$ and $ \mathbb{H}=(0,\cdots,
(R(\sigma, k)J(ik) F)_1)$.

The properties of $E$, $J(ik)$ and $P_{1}$ and the triangular structure  already give
$$ |u_{2}|_{s} \leq C_{s}\big(  |u_{1}|_{s+m_{k}-1 } + |F|_{l_{k}+s+ 1}\big).$$
Consequently, it suffices to prove that for every $s \geq 0$, 
 $$ |V|_{s} \leq  C_{s} | \mathbb{H} |_{s} $$
  where $V$ is the solution of the ODE \eqref{abst} to get
   the result.
    
 Let us denote by $T(\sigma, k , x,x')$  the fundamental solution of $V_{x} = \mathbb{A}V$ i.e. the solution such that 
$T(\sigma,k ,x',x')=Id$. Thanks to our assumption \eqref{asympt} on the behavior as
$|x|\rightarrow\infty$ of $A(\sigma, k , x)$, we can use classical
perturbative ODE arguments (more precisely the roughness of
exponential dichotomy, see \cite{Coppel} for example). Namely,  the equation
$V_{x} = AV$ has an exponential dichotomy on $\rit_{+}$ and $\rit_{-}$, i.e., there exists 
projections $P^+(\sigma, k , x)$, $P^-(\sigma, k, x)$ which are smooth in  the parameter
 $\sigma$  with the invariance property 
\begin{equation}\label{invar}
T(\sigma, k ,  x, x') P^\pm(\sigma, k , x') = P^{\pm}(\sigma, k  , x) T(\sigma, k , x, x')
\end{equation}
and  such that  there exists $C$ and $\alpha>0$ such that  for every $U \in \mathbb{C}^{N_k}$,  
and $\sigma  \in \mathcal{K}$, we have 
\begin{eqnarray*}
& &  |T(\sigma, k, x ,x') P^+(\sigma, k ,  x') U | \leq C e^{-\alpha(x- x')\ } | P^+(
\sigma, k ,  x') U|, \, 
x\geq x'\geq 0, \,  \\
& & |T(\sigma, k , x ,x')( I -  P^+(\sigma, k , x'))U| \leq C e^{\alpha(x- x')\ } |(I -  P^+(
\sigma, k , x'))U| , \, 
0 \leq x \leq x', \,  \\
& &  |T(\sigma, k  , x ,x') P^-(\sigma, k,x,   x') U| \leq C e^{\alpha(x- x')\ }| P^-(
\sigma, k ,  x') U| , \, 
x\leq x'\leq 0, \\
& &  |T(\sigma, k  , x ,x') (I - P^-(\sigma, k ,  x')) U| \leq C e^{-\alpha(x- x')\ }| (I - P^-(
\sigma, k ,  x') )U| , \, 
0 \geq x\geq x'.
\end{eqnarray*} 
In particular,  note that a solution $T(\sigma, k , x,0)V^0$ is decaying when $x$ tend to $\pm \infty$ if and only if
$V^0$ belongs to $\mathcal{R}(P^\pm(\sigma, k , 0))$.
Since when $\sigma$ is in $\mathcal{K}$, the Evans function does not
 vanish,   we have by definition  no non trivial solution decaying in both sides and hence we have 
\beq\label{inter}
\mathcal{R}( P^+(\sigma, k ,0) ) \cap \mathcal{R}(P^-(\sigma, k , 0)  )= \{ 0 \}.
\eeq
Let us choose bases  $(r_{1}^\pm, \cdots, r_{N^\pm}^\pm)$  of 
$ \mathcal{R}( P^\pm(\sigma, k ,0))$
(where $N^+ + N^- = N_k$)
which depends on $\sigma$ in   a smooth way (see \cite{Kato} for example) then we can define
$$ 
M(\sigma, k  ) = (r_{1}^+, \cdots,   r_{N^+}^+, r_{1}^-, \cdots, r_{N^-}^-)
$$ 
and we note that  $M(\sigma, k )$ is  invertible for $\sigma  \in \mathcal{K}$ because of \eqref{inter}.
With, these new notations, we note in passing that  the Evans function  can actually be  defined by 
$$ 
D(\sigma , k ) = \det M( \sigma , k ).
$$ 
This allows us to define a new projection $P(\sigma, k )$  by
$$ 
P(\sigma , k ) = M(\sigma, k )
\left( 
\begin{array}{cc} 
I_{N^+} & 0  
\\ 
0 & 0
\end{array} 
\right)
M(\sigma, k )^{-1}
$$
and next
$$
P(\sigma, k , x)= T(\sigma, k ,  x, 0) P(\sigma,k).
$$
The main interest of these definitions is that we have 
 $\mathcal{R}( P(\sigma, k ))= \mathcal{R}( P^+(\sigma, k ,0))$  and $\mathcal{R}( I- P(\sigma, k ))= 
 \mathcal{R}( P^-(\sigma, k ,0))$.  Therefore thanks to (\ref{invar}), we have for every $x$  that  
$\mathcal{R}( P(\sigma, k ,x))= \mathcal{R}( P^+(\sigma, k ,x))$ and  similarly that
$$\mathcal{R}( I - P(\sigma, k ,x)) = \mathcal{R}( P^-(\sigma, k ,x)).$$ 
Consequently, we have the estimates
\begin{eqnarray}
& &    \label{dich+}
|T(\sigma, k , x ,x') P(\sigma, k , x')| \leq C e^{-\alpha(x- x')\ } , \, 
x,\, x' \in \mathbb{R}, \, x \geq x',  \,  \, \forall \sigma  \in \mathcal{K}, \\
\label{dich-}
& & |T(\sigma, k , x ,x')(I -  P(\sigma, k , x'))| \leq C e^{\alpha(x- x')\ } , \, 
x,\, x' \in \mathbb{R}, \, x \leq x',  \,  \, \forall \sigma  \in \mathcal{K}.
\end{eqnarray}
By using this property, the unique bounded solution of \eqref{abst} reads by Duhamel formula
$$ 
V(x) = \int_{- \infty}^x T(\sigma, k  ,  x, x')P(\sigma, k , x') \mathbb{H}(x')\, dx'
- \int_{x}^{+ \infty} T(\sigma, k , x, x')(I - P(\sigma, k , x') )\mathbb{H}(x')\, dx'
$$
and hence, we get thanks to \eqref{dich+}, \eqref{dich-} that
$$ 
|V(x)| \leq C \int_{\mathbb{R}} e^{- \alpha| x - x'|} |\mathbb{H}(x') |\, dx'
$$
which yields by standard convolution estimates 
$$ 
|V|_{L^2} \leq C |\mathbb{H}|_{L^2},
$$

We next estimate higher order derivatives. Write 
$$ 
\partial_{x}^{s+1} V = \mathbb{A} \partial_{x}^sV + [ \partial_{x}^s, \mathbb{A}]V +\partial_{x}^s \mathbb{H}\,.
$$
By  considering $[ \partial_{x}^s, \mathbb{A}]V$ as
part of the source term and by using the Duhamel formula, we get 
$$ 
|V|_{H^s} \leq C |\mathbb{H}|_{H^s}.
$$
This yields 
$$  |u_{1} |_{s} \leq  C |F|_{l_{k}+ 1 }.$$
 This ends the proof of Lemma \ref{lemper}.

\subsection{Proof of Lemma \ref{lemloc}}

  We study again  the equation
$$
 \sigma   w=  J(ik)  L w + J(ik) S(ik)w + J(ik) F.
$$
 Again, we can apply $R(\sigma, k)$ to 
 get
 \beq
 \label{wrJ}
 \sigma  R(\sigma, k)w - R(\sigma, k) \big( J(ik) Lw  + J(ik) S(ik)w \big) = R(\sigma, k)J(ik)F.\eeq
  To solve \eqref{wrJ}, we use a  method  close  to the one used in \cite{Kreiss-Kreiss}
in a different context. The problem is that in estimates (\ref{dich+}),
(\ref{dich-}), we have that $\alpha\approx \rho(k,K)$ may degenerate for $k\sim 0$. The convolution estimate
$$
\|e^{-\alpha|x|}\star f(x)\|_{L^2}\leq \frac{C}{|\alpha|}\|f\|_{L^2} 
$$ 
gives the rate of degeneration. The strategy is to write the solution $w$ as a
sum of two pieces. The first piece satisfies the needed estimate thanks to the
1d assumption (hence no degeneration in the limit $k\rightarrow 0$), while the
second piece satisfies an equation of type (\ref{wrJ}) with a source term
vanishing as $|\mbox{Re }\mu(k,\sigma)|$ in the limit $k\rightarrow 0$. This exactly
compensates the singularity in the convolution estimate.

To be more precise, we seek the solution of \eqref{wrJ} under the form 
\beq\label{decwr} w = u  +  v \eeq
 where $u$ solves
\beq
\label{w1r}
\sigma R(\sigma, 0^+) u  - R(\sigma, 0^+) J(0) Lu = R(\sigma, 0^+) J(0) F
\eeq
 and hence $v$ solves
 \begin{eqnarray}
 \label{w2r}
& &\sigma  R(\sigma, k)v - R(\sigma, k) \big( J(ik) Lv  + J(ik) S(ik)v \big) 
  = \\
 & &  \nonumber  - \Big( R(\sigma, k)J(\sigma, k) S(ik)u  + \sigma  \big(R(\sigma, k) - R(\sigma, 0^+)\big)
  u  \\
  & & \nonumber \quad +  \big( R(\sigma, k) J(ik) - R(\sigma, 0^+) J(0) \big)Lu \Big)
   + \big( R(\sigma, k)J(ik) - R(\sigma, 0^+)J(0) \big)F :=H.
   \end{eqnarray}
 The main interest  of this manipulation is that the source term
  of \eqref{w2r} now vanishes thanks to \eqref{compensation} when $k\rightarrow 0$
   if  $A_{\infty}(\sigma, k)$ has an eigenvalue of vanishing real part.
   
   To solve \eqref{w1r}, we can choose $u$ as the solution of
    $$ \sigma u - J(0)Lu= J(0) F.$$
    Since  we assume that $D$ does not vanish  on $\mathcal{K}$, we can
     use Lemma \ref{lemper} to get
     \beq
     \label{estw1}
      |u|_{s} \leq C |F|_{s+q}.
      \eeq
  Thanks to the assumption  \eqref{compensation},  this implies that
   the source term in \eqref{w2r}  satisfies  the estimate
\beq\label{estH1r}
|H(\sigma, k)|_{s} \leq C \, \rho(k, K)\, |F|_{s+q + q_{1}}  
\eeq    
for some $q_{1} \geq 0$.
 To study \eqref{w2r}, we can use the block structure
  \eqref{factorisation} to get
$$  v_{2}= E(\sigma, k)^{-1} \big( P_{2}(\sigma, k) v_{1} + H_{2}\big), 
 \quad P_{1}(\sigma, k) v_{1}= H_{1}.$$
  Since by assumption the operators $E$ and $P_{2}$
   have a continuous extension to $\mathcal{K} \times [0, M]$, we get
$$ |v_{2}|_{s} \leq C \Big(  |v_{1}|_{s+ m- 1} +  |F|_{s+ l + q +  q_{1} }\Big) $$
 uniformly for $(\sigma, k) \in \mathcal{K} \times (0, M]$.
  Consequently, we  only need to study the equation
  $$ P_{1}(\sigma, k) v_{1}= H_{1}$$
   to get the result.
    As in the proof of Lemma \ref{lemper}, we  rewrite  this equation as 
     a first order system
\beq\label{Vr}
V_{x} = A(\sigma, k ,x) V + \mathbb{H}\,.
\eeq
To get the existence of exponential dichotomies for 
\beq\label{Vrhom} V_{x} = A(\sigma, k ,x) V \eeq
on $\mathbb{R}_{+}$  and $\mathbb{R}_{-}$ when $k \neq 0$ with  a good control
of $C$ and $\alpha$, we can use the conjugation Lemma of \cite{Metivier-Zumbrun}. 
Thanks to Lemma~2.6 of \cite{Metivier-Zumbrun}, there exist  conjugators $\mathcal{W}_{\pm}(x,\sigma, k )$
such that $\mathcal{W}_{\pm}(x,\sigma, k )$ are  invertible  for every $(\sigma, k)$ with
$(\sigma, k )$ with $\mbox{Re }\sigma>0$, $k \in [0, M]$   and $x \in \mathbb{R}_{\pm}$ with a uniform
bound of $\mathcal{W}_{\pm}$ and $\mathcal{W}_{\pm}^{-1}$  and  the property
$$ 
\mathcal{W}_{\pm} = Id + \mathcal{O}( e^{ - \pm \alpha x })
$$
when $x$ tends to $\pm \infty$. Moreover,  for every $V$ solution of
\eqref{Vrhom}, $V_{1} = \mathcal{W_{\pm}}^{-1} V $ solves 
\beq\label{asympt_bis} (V_{1})_{x} = \mathbb{A}_{\infty}(\sigma, k ) V_{1}.\eeq
 Since for $k \neq 0$ , the spectrum of
  $A_{\infty}(\sigma, k)$ does not intersect the imaginary axis,  the
  autonomous system \eqref{asympt_bis} 
has an exponential
dichotomy on  $\mathbb{R}$, for $ k \neq 0$. Namely, there exists $P_{\infty}(\sigma, k ) $  and $C>0$ such that
\begin{eqnarray}\label{+}    
| e^{x  A_{\infty}(\sigma,k) } P_{\infty}  U | &  \leq  & C e^{ - \alpha(k) x    }
|U|,\, 
\forall\, x \geq 0,   \, \forall\, U \in \mathbb{C}^N \\
\label{-}  | e^{x  A_{\infty}(\sigma, k ) } ( I -  P_{\infty})  U | &  \leq  & C e^{  \alpha(k) x     } |U|
 ,\, \forall\, x \leq 0,   \, \forall\, U \in \mathbb{C}^N\,.
\end{eqnarray}
where we can take  $\alpha ( k) =\rho(k,K) /2$. Moreover, $P_{\infty}$ can be
continued up to $k=0$.  Thanks to the conjugation property,  we have 
\beq\label{conjug}
T(\sigma,k , x)  = \mathcal{W}_{\pm}(\sigma, k ,x) e^{ x A_{\infty}(\sigma, k ) }
\mathcal{W}_{\pm}(\sigma, k ,0)^{-1},\,  x \in \mathbb{R}_{\pm}
\eeq
and  hence the projections $P_{\pm}(\sigma, k ,0)$  which  define the
exponential dichotomy for \eqref{Vrhom} are given by 
$$ 
P_{+}(\sigma, k ,0) =   \mathcal{W}_{+}(\sigma, k ,0) P_{\infty} \mathcal{W}_{+}(
\sigma, k , 0) ^{-1}, \quad 
P_{-}(\sigma, k ,0) =   \mathcal{W}_{-}(\sigma, k ,0)( Id -  P_{\infty} ) \mathcal{W}_{-}(
\sigma, k , 0)^{-1}. 
$$
Since by assumption,   the Evans function
 $ \tilde{D}$  does not  vanish up to $k=0$, we still
have that
$$ 
\mathcal{R} P_{+}(\sigma, k ,0) \oplus \mathcal{R} P_{-}(\sigma, k ,0) = \C^N.
$$
Thanks to \eqref{+}, \eqref{-} and \eqref{conjug},  we  thus get that
\eqref{dich+}, \eqref{dich-} are still true for $ \sigma \in \mathcal{K}$ and
$ |k | \leq M$, $k \neq 0$ with $C$ independent of $k$ and $\alpha =  \alpha (k)$.
          
By using again Duhamel formula and convolution estimates, we  get  for the solution of \eqref{Vr}
$$ 
|V|_{s } \leq \frac{C}{\alpha (k)} |\mathbb{H}|_{s}
$$
and hence, we can use \eqref{estH1r} to get
$$ 
|v|_{s} \leq \frac{ C}{ \alpha(k) } |H|_{s}  \leq { C \rho(k, K) \over \alpha(k) }\, |F|_{s+q+q_1}
=C|F|_{s+q+q_1}
$$
for $k \neq 0$. This ends the
proof.        


\section{Criterion for the existence  of multipliers}
\label{multip}
In this section we prove a criterion for the assumption of Section~\ref{Ms}.
\begin{lem}
\label{lemmult}
Suppose that for  every $s\geq 2$ there exists a  symmetric operator $K_{s}$, 
bounded on $L^2$  such that
$$
E_{s}\equiv-\frac{1}{2}\partial_{x}  [J(ik), R ] \partial_{x}-\frac{s}{2} 
\big( \partial_{x} J(ik) \, [\partial_{x}, R ] + [\partial_{x}, R ]^* J(ik)
\partial_{x} \big)
+\frac{1}{2}[ K_{s}, J(ik) L_{0} ] 
$$
is an operator of order $1$, i.e. there exists $C_{s}(k)>0$ with 
$
 |E_{s}u  | \leq C_{s}(k) |u|_{1}.
$
Then, we have that there exists $M_s$ such that (\ref{Ms1}) and (\ref{Ms2}) hold.
\end{lem}

 This general criterion can be  used in a very simple way when
 $J(ik)$ is a zero order operator, i.e. $J(ik) \in \mathcal{B}(L^2)$.
  Indeed,  we notice that in such a situation the second term in $E_{s}$
   is already a first  order operator. We  will prove  the following
    corollary:
\begin{corollaire}
\label{cormult}
 Assume that $J(ik)\in \mathcal{B}(L^2)$ and that
  $L_{0}= - \partial_{x}^2 + \tilde{L}$ with $\tilde{L}\in \mathcal{B}(H^1, L^2)$.
   Then $K_{s}=R$ verifies the assumption of Lemma
    \ref{lemmult} i.e. $E_{s}$ is a first order operator and
     hence  there exists $M_{s}$ such that \eqref{Ms1}, \eqref{Ms2} hold. 

\end{corollaire}

\subsection{Proof of Corollary \ref{cormult}}
We  check that the assumption of Lemma \ref{lemmult} is verified with  $K_{s}=R$.
 As already noticed the second term in the definition of $E_{s}$
  in Lemma \ref{lemmult} is already
  a first order operator since $J$ is a zero order operator.
   Next, by the assumption $L_{0}= - \partial_{x}^2 + \tilde{L}$, we notice that
  $$ [K_{s}, J(ik) L_{0}]=  - \partial_{x} [K_{s}, J(ik) ] \partial_{x} + \tilde{E}$$
   with $\tilde{E}$ a first order operator. This proves that
    $E_{s}$ is indeed a first order operator with the choice $K_{s}=R$.

\subsection{ Proof  of Lemma \ref{lemmult} }
For $s \geq 2$,  we  define the symmetric operator
$$ 
M_{s} u\equiv (-1)^{ s} \partial_{x}^{2s} u + ( - 1)^{s-1}\partial_{x}^{s-1}\big( K_{s}\partial_{x}^{s-1} u \big).
$$ 
Thanks to the $L^2$ boundedness of $K_s$ the assumption (\ref{Ms1}) is clearly satisfied.
Let us next check (\ref{Ms2}).
For that purpose, we need to evaluate the quantity
$$
\Re \big( (J(ik)(L + S(ik)) u, M_{s} u).
$$
Since $J(ik)L_{0}$ is skew-symmetric, we have
\beq\label{eL2}
\Re 
\big( J(ik) L_{0}u,  ( -1)^{s} \partial_{x}^{2s} u \big)
 = \Re\big( J(ik) L_{0} \partial_{x}^{s}u, \partial_{x}^{s}u \big) = 0.
\eeq
We can also write
\begin{multline*}
\Re\big( J(ik) R\, u, (-1)^{s}\partial_{x}^{2s} u \big) 
= \Re \big(J(ik) R\, \partial_{x}^{s}u, \partial_{x}^{s}u \big
 ) 
\\
+ 
s \Re\big( J(ik)  [\partial_{x}, R ] \partial_{x}^{s-1} u, \partial_{x}^{s} u
\big)   + 
(\mathcal{C}u, \partial_{x}^{s}u)
 \end{multline*}
 where 
$| \mathcal{C}u | \leq C |u|_{s-1}.$
Furthermore, since $J(ik)$ is skew symmetric and $R$ symmetric, 
\begin{eqnarray*}
 \Re\big(J(ik) R\, \partial_{x}^{s}u, \partial_{x}^{s}u \big)
 & = &  \frac{1}{2}\Re
\big((J(ik) R- RJ(ik)) \partial_{x}^{s}u, \partial_{x}^{s}u\big) 
\\
& = &  
-\frac{1}{2}\Re\big( \partial_{x}\, [ J(ik), R ] \, \partial_{x} \, \partial_{x}^{s-1}
u, \partial_{x}^{s-1} u \big).
\end{eqnarray*}
Therefore, we finally find
\begin{eqnarray}
\label{eL3}
& & \Re \big( J(ik) R\, u, (-1)^{s}\partial_{x}^{2s} u \big) 
\\
\nonumber 
& & =  - \frac{1}{2}\, \Re\big( \partial_{x}\, [ J(ik) , R ] \, \partial_{x} \, \partial_{x}^{s-1}
u, \partial_{x}^{s-1}u \big) \\
\nonumber     
& & 
-\frac{s}{2}\Big(  \big( \partial_{x}\, J(ik) \,  [\partial_{x} , R ] +  [\partial_{x}, R]^*
\, J(ik) \, \partial_{x} \big) \partial_{x}^{s-1} u, \partial_{x}^{s-1} u \Big)
       + \mathcal{O}(1) | u |_{s}\, |u |_{s-1}.
\end{eqnarray}
Next, since $J(ik) S(ik)$ is skew-symmetric, 
\beq
\label{eL4}
\Re
\big( J(ik) S(ik) u, (-1)^{s} \partial_{x}^{2s } u \big)
 =  \big( J(ik) S(ik) \partial_{x}^{s} u,  \partial_{x}^{s } u \big) = 0.
\eeq
Since $J(ik)L_{0}$ is skew-symmetric and $K_{s}$ symmetric,  we also have 
 \begin{eqnarray}
 \nonumber 
 \Re \big(  J(ik) L_{0} u, (-1)^{ s-1 } \partial_{x}^{s-1 }  \, K_{s}\,  \partial_{x}^{s-1}\, u \big)
  &= &   \Re \big(  K_{s}\, J(ik) L_{0} \,   \partial_{x}^{s-1} u,  \partial_{x}^{s-1 } u \big)
  \\
  \label{eL5}
   & = & \frac{1}{2}
      \big(  [ K_{s},  J(ik) L_{0} ] \,   \partial_{x}^{s-1} u,  \partial_{x}^{s-1 } u \big).
     \end{eqnarray}
Since $K_s$ is bounded on $L^2$ and $J(ik)$ of order one, we have that
\begin{equation}\label{el6}
\Re \big(  J(ik) R u, (-1)^{ s-1 } \partial_{x}^{s-1 }  \, K_{s}\,
\partial_{x}^{s-1}\, u \big)= 
\mathcal{O}(1) | u |_{s}\, |u |_{s-1}.
\end{equation}
Next, since  by the  assumptions of section  \ref{sS}, we have in particular that
$J(ik) S(ik) \,\partial_{x}\in \mathcal{B}(H^2, L^2 )$ and since $K_{s}$ is bounded on $L^2$, we  have
\begin{multline}\label{el7}
\Re \big(J(ik) S(ik) u, (-1)^{s-1} \partial_{x}^{s-1} K_{s} \partial_{x}^{s-1} u
\big) =  
\\
\Re \big(  J(ik) S(ik) \partial_{x} \, \partial_{x}^{s-2}\, u,   K_{s} \partial_{x}^{s-1} u \big)
 = \mathcal{O}(1) |u |_{s}\, |u |_{s-1}.
\end{multline}
Collecting \eqref{eL2}, \eqref{eL3}, \eqref{eL4},  \eqref{eL5}, \eqref{el6}, \eqref{el7}, we infer
that
$$
\Re \big( (J(ik)(L + S(ik)) u, M_{s} u)=
(E_s\partial_x^{s-1}u, \partial_{x}^{s-1}u).
$$
In view of the assumption on $E_s$, we obtain that the assertion of the proof
of the lemma holds.

\section{Proof of Theorem \ref{theoper} (periodic perturbations)}
The  general strategy of the proof is  inspired from the work of Grenier \cite{Grenier} in fluid mechanics.
\subsection{Construction of a most unstable eigenmode}
By the assumption there exists an unstable mode
with associated transverse frequency $k_0\neq 0$.
The first step of the proof is to find the most unstable eigenmode.
This means that we look for an unstable mode with  associated transverse frequency $mk_0$,
$m\in\Z$ such that the associated amplification parameter $\sigma$ has maximal
real part. This is indeed possible thanks to the following lemma.
\begin{lem}\label{krou}
Consider the problem
\beq\label{eqvp}
\sigma U = J(imk_{0}) ( LU + S(im k_{0}) U),\quad U\in L^2(\R;\C^d)\,.
\eeq
There exists $K>0$ such that for $|mk_{0}| \geq K$  
there is no nontrivial solution of (\ref{eqvp}) with $ \Re(\sigma) \neq 0$.
\\
In addition, for every $k\neq 0$ there is at most one unstable mode with
corresponding  transverse frequency $k$.
\end{lem}
\begin{proof}
Recall that by  assumption  if $U$ solves (\ref{eqvp}) then $U$ belongs to
$H^{\infty}(\R;\C^N)\cap\mathcal{D}_{S}$. By taking the real part of the scalar product of (\ref{eqvp}) 
with $LU + S(im k_{0}) U$,  we get  the following "conservation law"
\beq\label{rienkgrand}
0=\Re(\sigma ( (U, LU) + ( U,  S(imk_{0})U))) = \Re(\sigma) ( (U, LU) + ( U,  S(imk_{0})U) ).
\eeq
Indeed, since $J(imk_{0})$ is skew-symmetric, we have
$\mbox{Re } (J(imk_{0})u, u)=0$  for every $u \in H^\infty$ and we have also used
 that $L$ and $S(ik)$ are symmetric.
Thanks to \eqref{klarge}, we get that for $|mk_{0}| \geq K$  
there is no nontrivial solution of (\ref{rienkgrand}) with $ \Re(\sigma)\neq 0$.

   Let us now prove the second assertion of the lemma, i.e.
we shall prove that  for $k \neq 0$ there is at most one unstable eigenmode with
corresponding  transverse frequency $k$.
Thanks to \eqref{eqvp}, we first notice that an unstable eigenmode must be in
the image of $J(imk_{0})$, consequently, since $J(i k )$ is into, we can write $U= J 
(ik)V$ with
$V\in H^{\infty}(\R;\C^N)$ a nontrivial solution of
\beq\label{eqV}
\sigma J(ik) V = \big( J(ik)LJ(ik)  + J(ik) S(ik) J(ik) \big)V \equiv M_{k} V, \quad k= m k_{0}.
\eeq
Note that $M_{k}$ is a symmetric operator.
Next, we observe that the operator $J(ik)LJ(ik)$ has at most one  positive  eigenvalue.
Indeed, by contradiction, if    $J(ik)LJ(ik)$ had an invariant subspace $E$  of dimension at least
$2$ on which the quadratic form $J(ik)LJ(ik)$ is positive definite, then the quadratic
form $(Lu,u)$ would be negative 
definite on $J(ik)E$ and since $J(ik)$ is into $J(ik)E$ is also two-dimensional. This  gives
a contradiction  since $L$ has only one simple negative eigenvalue.
\\
Next, we can also prove  that $M_{k}$ has at most one  simple positive
eigenvalue. Again, if $M_{k}$ has an invariant subspace $E$ of dimension at
least $2$ on which the
quadratic form $(M_{k}u, u )$ is positive definite then there exists   
$u \in E \cap (\psi)^\perp \neq \{0\} $ where $\psi$ is the only positive eigenvalue of  
$J(ik)LJ(ik)$. Since on $ (\psi)^\perp$, $J(ik)LJ(ik)$ is non positive, and $S(ik)$ is positive, we get
$$ 
(M_{k} u, u ) = (J(ik)LJ(ik) u, u ) + (J(ik)S(ik)J(ik)u, u ) \leq 0
$$
which yields a contradiction. Consequently $M_{k}$ has at most one  positive eigenvalue.
Finally, we can use \cite[Theorem~3.1]{PW} to get that $J(ik)^{-1} M_{k}$ has at most one unstable eigenvalue.
Consequently, for $k \neq 0$, there is at most one unstable  $\sigma$ for which $\eqref{eqvp}$ has a nontrivial solution.
This completes the proof of Lemma~\ref{krou}.
\end{proof}         
By  the assumption \eqref{evans1D}, we know that for $k=0$, there is no unstable eigenmode.
We consider the finite set $A$ of integers $m$ such that $k_{0} \leq |\, mk_{0}| \leq K$,
where $K$ is provided by Lemma~\ref{krou}. Again by Lemma~\ref{krou}, for
every $m\in A$ there is at most one unstable mode with corresponding
transverse frequency $mk_0$. Moreover, by the assumption of
Theorem~\ref{theoper}, for $m=1$ there is an unstable mode. 
We now take the unstable mode $U$ corresponding to $m_0\in A$ with maximal real part of the
corresponding amplification parameter which we note by $\sigma_0$. We set
$$
u^{0}(t,x,y)\equiv e^{\sigma_0 t } e^{ i m _0k_{0} y }\,
U+e^{\overline{\sigma_0} t } e^{- i m _0k_{0} y }\, \overline{U}
=2\Re\big(e^{\sigma_0 t } e^{ i m _0k_{0} y }\,U\big)\,.
$$
To prove Theorem \ref{theoper}, we shall use $Q+\delta u^0(0)$ as an initial data
for \eqref{ham2}.
Thanks to our  assumptions of section~\ref{asnon} about the nonlinear problem,
the problem \eqref{ham2} is locally well-posed
with data $Q+\delta u^0(0)$.
\subsection{Construction of an high order unstable approximate solution}
Denote by $F_j\in C^{\infty}(\R^d;\R)$, $1\leq j\leq d$ the derivative of $F$ with respect
to the $j$'th variable, i.e. $\nabla F=(F_1,\cdots,F_{d})$. For
$\alpha\in\N^d$, we set 
\begin{equation}\label{falpha}
F_{\alpha}\equiv \big(\partial^{\alpha}F_{1}(Q),\cdots,\partial^{\alpha}F_{d}(Q)\big).
\end{equation}
Let us look for a solution of \eqref{ham2} under the form $u=Q+\delta v$,
where $\delta\in ]0,1]$. Recall the Taylor formula
$$
f(x+y)-f(y)=\sum_{1\leq |\alpha|\leq N}\frac{ x^\alpha  }{\alpha!}\partial^{\alpha}f(y)+(N+1)
\sum_{|\alpha|=N+1}\frac{x^{\alpha}}{\alpha!}\int_{0}^{1}(1-t)^{N}\partial^{\alpha}f(tx+(1-t)y)dt,
$$
where $N\geq 1$ and $f\in C^{\infty}(\R^d;\R)$.
In what follows, we shall also use  that for $s \geq 2$,   $\mathbb{H}^s$  is an algebra , and that 
$
\|f(u)\|_{s}\leq \Lambda (\|u\|_{s}),
$
where $\Lambda:\R\rightarrow\R^{+}$ is a continuous function. 
We obtain thus that for every $M\geq 1$, $v$ solves the equation 
\beq\label{nonlineq}
\delta \partial_{t} v  = \mathcal{J}(\partial_y)\Big( 
\delta (L+ \mathcal{S}(\partial_{y}))v+
\sum_{2\leq |\alpha|\leq M+1}\delta^{|\alpha|}v^{\alpha}F_{\alpha}+\delta^{M+2}R_{M,\delta}(v)
\Big),
\eeq
where $F_{\alpha}$ is defined by (\ref{falpha}) and $R_{M,\delta}$ satisfies
for $s\geq 2$
$$
\forall\, \delta\in ]0,1],\quad \forall v\in \mathbb{H}^s\,,\,
\|R_{M,\delta}(v)\|_{s}\leq \|v\|_{s}^{M+2}\Lambda_{M}(\|\delta v\|_{s})\,,
$$
where $\Lambda_{M}:\R\rightarrow\R^{+}$ is a continuous function. 
We define $V_{K}^s$ as the space 
$$ 
V_{K}^s =\Big\{ u\,:\, \, u= \sum_{j=- K}^{K} u_{j}\, e^{  i  jm_{0}k_0  y  },\quad \, u_{j} \in H^s(\R) \Big\} 
$$
and we define a norm on $V_{K}^s$ by
$
|u|_{V_{K}^s} = \sup_{j} |u_{j}|_{s}.
$ 
Let us notice that  $u^{0} $ is such that $u^{0} \in V^{s}_{1}$ for all $s\in\mathbb{N}$.
Following the strategy of \cite{Grenier}, for $s\gg 1$, we look for an high order solution under the form 
\begin{equation}\label{u_app}
u^{ap}= \delta u^0  + \sum_{ k=2}^{M+1} \delta^k u^k , \quad u^k \in  V^{s}_{ k +1}
\end{equation}
such that $u^k_{/t=0}=0$ and $M\geq 1$ is to be fixed later.

By plugging the expansion in \eqref{nonlineq}  and by 
cancelling the terms involving $\delta^{k+1}$, $1\leq k\leq M$,
we choose $u^k$ so that $u^k$ solves the problem
\beq\label{uk}
\partial_{t} u^k = \mathcal{J}(\partial_y) ( L u^k + \mathcal{S}(\partial_{y}) u^k)
 + \mathcal{J}(\partial_y) \sum_{2\leq |\alpha|\leq k+1} \Big( \sum_{|\beta|= k+1-|\alpha|}
u^{\beta_1}_{1}\cdots u^{\beta_d}_{d}\Big)
F_{\alpha}
\eeq
where $u^{\beta_j}_{j}$ stands for the $j$'th coordinate of $u^{\beta_j}$ and
with the initial condition $u^k_{/t=0} = 0$.
Note that the term involving $\delta$ cancels thanks to the choice of $u^0$
(while the term in front of $\delta^0$ is absent in \eqref{nonlineq} thanks to
the choice of $Q$).
Thanks to our assumptions $u^k$ is a solutions of a linear equation which is
globally defined. Indeed, we can define $\exp(JL_0)$ via the Fourier transform
and then treat the problem for $u^k$ perturbatively.
Moreover $u^k\in V^s_{k+1}$ for every $s\in\R$.
The main point in the analysis of $u^{ap}$ is the following estimate.
\begin{proposition}\label{propuk}
Let us fix an integer $M\geq 1$.
Let $u^k$ be the solution of \eqref{uk},  $0\leq k\leq M$.
Then for every integer $s\geq 1$ there exists a constant $C_{M,s}$ such that we have the bound
\beq\label{estuk} 
|u^k(t)|_{V^{s }_{k+1}} \leq C_{M,s}e^{ (k+1 ) \Re(\sigma_{0})t}, \quad \forall\, t \geq 0.
\eeq 
As a consequence there exists $G\in {\mathbb H}^s$ for all $s$ such that 
$$
\partial_{t}(Q+u^{ap})-\mathcal{J}(\partial_y)\big(L_0(Q+u^{ap})+\nabla
F(Q+u^{ap})+\mathcal{S}(\partial_{y})(Q+u^{ap})\big)=
\mathcal{J}(\partial_y)G
$$
and  for $0\leq t\leq \log(1/\delta)/\Re(\sigma_0)$ and $s\geq 0$ one has the bound
$$
\|{\mathcal J}(\partial_y)G(t)\|_{s}\leq  C_{M,s}\delta^{M+2}\,e^{ (M+2 ) \Re(\sigma_{0})t},
$$
where $C_{M,s}$ is independent of $t \in [ 0, \log(1/\delta)/\mbox{Re }\sigma_{0}]$ and $\delta\in ]0,1]$. 
\end{proposition}
By an easy induction argument Proposition~\ref{propuk} is a consequence of the following statement.
\begin{proposition}\label{theolinsource}
There exists $q\in\N$ such that for $s\geq 3$,  if
$f(t) \in V_{K}^{s+q}$ satisfies
\beq\label{F}
|f(t)|_{V_{K}^{s+q}} \leq C_{K,s} e^{ \gamma t}, \quad \gamma \geq 2\Re(\sigma_0)
\eeq 
then the solution $u$ of the linear problem
\beq\label{linsource}
\partial_{t} u   = \mathcal{J}(\partial_y)(Lu + \mathcal{S}(\partial_{y})u ) +
\mathcal{J}(\partial_y) f, 
\quad u_{/t=0}= 0.
\eeq
belongs to $ V_{K}^{s}$ and satisfies the estimate
\beq\label{ulin}
|u(t)|_{V_{K}^{s}} \leq \tilde{C}_{K, s} e^{ \gamma t}, \quad \forall t \geq 0.
\eeq 
\end{proposition}

Since $f$ has a finite number of  Fourier modes in $y$,
Proposition~\ref{theolinsource} is a direct  consequence of the following $1d$ result.
\begin{proposition}\label{1d}
There exists $q\in\N$ such that for $j$ such that $j/(m_{0}k_0)\in\Z$ with $|j|/(m_{0}
k_{0}) \leq K$ (where $K$ is provided by
Lemma~\ref{krou})
and  $s\geq 3$,  if we suppose also that
\begin{equation}\label{Fbis}
|f_j(t)|_{s+q}\leq C_{j,s}e^{\gamma t},\quad \gamma \geq 2\Re(\sigma_0)
\end{equation}
then the solution of
\beq\label{ulink}
\partial_{t}v  = J(ij)\big( L + S(ij) \big) v  +   J(ij) f_{j},  \, \quad v_{/t=0}= 0 
\eeq
satisfies 
$$
|v(t)|_{s}\leq C_{j,s}e^{\gamma t}\,.
$$
\end{proposition}
 We shall prove below  that Proposition~\ref{1d}  is  a consequence of the following key resolvent estimate.
\begin{proposition}[{\bf Resolvent Estimates}]\label{resolvant}
Let $\gamma_{0}$ be such that  $Re(\sigma_0) <\gamma_{0}<\gamma$.
Suppose that $w$ solves the resolvent equation
\beq\label{w}
(\gamma_{0} + i \tau ) w =  J(ij)\big( L + S(ij) \big) w  + J(ij)  H
\eeq
with $|j|/(m_{0}k_{0})\leq K$.
Then there exists $q\in\N$ such that for $s\geq 1$ an integer there exists 
$C(s,\gamma_{0}, K)>0$ such that  for every $\tau$, we have the estimate
\beq\label{estresolvant}
|w(\tau)|_{s} \leq C(s,  \gamma_{0}, K) |H(\tau)|_{s+q}.
\eeq
\end{proposition}
\subsubsection{Proposition~\ref{resolvant} implies Proposition~\ref{1d}}
For $T>0$, we introduce $G$ such that
$$ 
G=0, \, t<0, \quad G=0, \, t>T, \quad G= f_j,  \, t \in [0, T]
$$ 
and we notice that the solution of 
$$
\partial_{t} \tilde{v} =  J(ij)\big( L + S(ij) \big) \tilde{v}  + J(ij) G, \quad \tilde{v}_{/t=0}=0
$$
coincides with $v$ on $[0,T]$. Indeed, $w = \tilde{v}- v$ solves for $t\in
[0,T]$ the equation
$$ \partial_{t} w = J(ij)\big( L + S(ij) \big) w, \quad w_{/t=0} = 0. $$
By  taking the real part of the  scalar product of this equation with $w$,
 we get  by skew-symmetry of  $J(ij)L_{0}$,  $J(ij)S(ij)$ and $J(ij) $ that  : 
 \begin{eqnarray*}
  \frac{d}{dt}  |w|^2 & = & 2 \mbox{ Re }\big(J(ij)Rw, w \big)
   =  - \big( w, [R, J(ij) ] w \big).
\end{eqnarray*}
Consequently, thanks to \eqref{JR}, we get 
$$ 
\frac{d}{dt} |w|^2  \leq  C 
|w|^2
$$
and hence, after integration in time , we find that $w= 0$ on $[0, T]$. It is
therefore sufficient to study $\tilde{v}$. Next, we set
$$ 
w(\tau,x)= \mathcal{L} \tilde{v}(\gamma_{0}+ i \tau), \quad 
H(\tau, x) = \mathcal{L}G ( \gamma_{0} + i \tau ), \quad (\tau, x) \in \mathbb{R}^2
$$
where $\mathcal{L}$ stands for the Laplace transform in time :
$$ 
\mathcal{L}f(\gamma_{0}+ i  \tau) = \int_{0}^{\infty} e^{- \gamma_{0} t - i \tau\, t} f(t)\, dt.
$$
By using Proposition~ \ref{resolvant} and Bessel-Parseval identity, we get that for every $T>0$, 
\begin{eqnarray*}
& & \int_{0}^T e^{-2 \gamma_{0} t } |v(t)|_{s}^2\, dt
\leq \int_{0}^{+ \infty} e^{-2 \gamma_{0} t } |\tilde{v}(t)|_{s}^2\, dt
= C \int_{\mathbb{R}} |w(\tau)|_{s}^2 \, d\tau \\
& & \leq C \int_{\mathbb{R}} |H(\tau)|_{s+q}^2 \, d\tau
= \int_{0}^{T} e^{-2 \gamma_{0} t } |f_j(t)|_{s+q}^2\, dt
\end{eqnarray*}
and finally thanks to \eqref{F}, we get
\beq\label{estsource}
\int_{0}^T e^{-2 \gamma_{0} t } |v(t)|_{s}^2\, dt\leq C \int_{0}^T e^{2( \gamma - \gamma_{0})t}\, dt
\leq C e^{ 2( \gamma - \gamma_{0})T}
\eeq
since $\gamma_{0}$ was fixed  such that $\gamma> \gamma_{0}$.
To finish the proof, we shall use  a crude  $H^s$ estimate 
for  the equation \eqref{ulink}. By 
 using that $J(ij)L_{0}$ and $J(ij)S(ij)$ are skew-symmetric  together with  \eqref{Fbis}, 
  we obtain
\begin{eqnarray*}
\frac{ d}{ dt} |v(t)|_{s}^2 & \leq & C \Big(  
|f_j(t)|_{s+1}^2 + 2 \mbox{Re } 
\sum_{ |\alpha |\leq s }\big(J(ij) \partial_{x}^\alpha( Rv), \partial_{x}^\alpha v \big) \Big) \\
& \leq &
C |v(t)|_{s+1}^2 +Ce^{2\gamma t}\,,
\end{eqnarray*}
where we have used that $J(ij)$ is an operator of order $1$.
It is possible to have a better estimate involving only $|v|_{s}^2$
 in the right-hand side, but it is useless here.
Next, for $0<\gamma_0 <\gamma,$ we get 
$$
\frac{ d}{  dt}\Big( e^{-2 \gamma_{0} t } |v(t)|_{s}^2\Big)
\leq 
C \Big( e^{- 2 \gamma_{0}t } |v(t)|_{s+1}^2  + e^{2 (\gamma - \gamma_{0})t}\Big).
$$
 Therefore, we can integrate  in time and use \eqref{estsource} (with $s+1$ instead of
$s$) and   the fact that $\gamma>\gamma_{0}$, to find 
$$ 
e^{- 2 \gamma_{0} t } |v(t)|_{s}^2  \leq C e^{2 (\gamma - \gamma_{0}) t }. 
$$
Therefore, we have shown  that Proposition~\ref{resolvant} implies Proposition~\ref{1d}.
\subsection{Proof of Proposition~\ref{resolvant}}
We shall deal differently with the large and bounded temporal frequencies. 
Indeed, Proposition~\ref{resolvant} is a consequence of the following two statements.
\begin{lem}\label{HF}
 For every $\gamma_{0}>0$  and $K \in \mathbb{N}$, 
there exists $M>0$  such that  for every $s \geq 1$, 
 there exists  $C(s, \gamma_{0}, K)$ such that
for $|\tau|\geq M$, $s\geq 1$, we have the estimate
\beq\label{estHF}
|w(\tau)|_{s}^2 \leq C(s, \gamma_{0}, K) |H(\tau)|_{s+1}^2.
\eeq
\end{lem}
\begin{lem}\label{BF}
There exists $q\in\N$ such that 
for every $\gamma_{0}$,  $\mbox{Re } \sigma_{0}< \gamma_{0}< \gamma$,
 $s \geq 1$, $K \in \mathbb{N}$ and $M \geq 0$, there exists 
  $C(s, \gamma_{0}, K, M)$ such that 
for $|\tau| \leq M$ and $s\geq 1$,  we have the estimate
\beq\label{estBF}
|w(\tau)|_{s}^2 \leq C(s,\gamma_{0}, K, M) |H(\tau)|_{s+q}^2.
\eeq
\end{lem}
\subsubsection{Proof of Lemma \ref{HF}}
We first prove \eqref{estHF} for $s=1$.       
Recall that the equation \eqref{w} reads as follows  
 \beq\label{wL}
(\gamma_{0} + i \tau ) w  = J(ij) (L + S(ij) ) w + J(ij) H.
\eeq
By  the assumption \eqref{spectre},
we can define an orthogonal decomposition: 
\beq\label{dec}
w = \alpha  \varphi_{-1} +  w_{0} + w_{\perp}
\eeq
where 
\beq\label{decprop}
L \varphi_{-1} = \mu \varphi_{-1}, \, \mu <0, \,  L w_{0} = 0, \, 
( Lw_{\perp}, w_{\perp}) \geq c_{0} |w_{\perp}|_{1}^2, \quad c_{0}>0.
\eeq
Indeed, to obtain  the last estimate, we   have used  that by the  assumption
 \ref{spectre}, we have  the  lower bound
\beq
\label{perpL2}
 (Lw_{\perp}, w_{\perp}) \geq c_{0}|w_{\perp}|^2,
 \eeq
but thanks to 
the decomposition $L=L_0+R$, and the lower bound
 \eqref{coerc}, we also have
\beq
\label{perpH1}
  (L w_{\perp}, w_{\perp}) \geq C^{-1}  |w_{\perp}|_{1}^2 - C |w_{\perp}|^2\eeq
 for some $C>0$ since $R$ is bounded on $L^2$.
  Consequently, we can consider   $ A\eqref{perpL2} + \eqref{perpH1}$
   with $A$ such that $Ac_{0}>C$ to get the claimed in (\ref{decprop}) lower bound.
    We normalize $\varphi_{-1}$ such that
     $ |\varphi_{-1}|=1$
    
Note that   by  the assumption after \eqref{spectre},   $\varphi_{-1}$ and $w_{0}$ are smooth.
Indeed, $w_{0}$ is smooth since the  kernel of $L$  is spanned by a finite
number of smooth eigenvectors and by expaning $w_{0}$
 on a smooth basis, 
we  also have that for every $s \geq 1$, there exists $C_{s}$ such that
\beq\label{w0}
|w_{0}|_{s} \leq  C_{s}|w_0|_2\leq C_s |w|_{1}.
\eeq
Again, we use the conservation law
\beq\label{conservation}
\gamma_{0}\big( (w, L  w) + ( w,  S(ij) w)\big) = \Re \big( ( J(ij) H,( L + S(ij))  w) \big).
\eeq
Consequently, we can use  \eqref{dec}, \eqref{decprop} to get
\begin{equation}\nonumber
 \gamma_{0}\big(\mu\,\alpha^2 |\varphi_{-1}|_{1}^2+c_{0} |w_{\perp}|_{1}^2
+ | w |_{S(ij)}^2 \big)
\leq 
  |(J(ij) H, S(ij) w ) |  +  |(J(ij)H, Lw) | .
\end{equation}
To  estimate  the right hand side, we first use 
 \eqref{JSw}  and the skew-symmetry of $J$ to get 
$$     |(J(ij) H, S(ij) w )|= | ( H, J(ij) S(ij)w ) | \leq C(j) |H|\, |w|_{S(ij)}.$$
  Next,  we notice that thanks to   \eqref{coerc} and  Cauchy-Schwarz, we have
  \beq
  \label{comut}
   |(u, L_{0}v) | \leq C |u|_{1}\, |v|_{1}\, \quad \forall u, \, v.
  \eeq
    This yields thanks to \eqref{Jik} : 
$$ |(J(ij)H, L w) | \leq  |J(ij)H|_{1} \, |w|_{1} \leq C  |H|_{2}\, |w|_{1}.$$
 We have thus proven that 
\begin{equation}\label{young0}
 \gamma_{0}\big(\mu\,\alpha^2 |\varphi_{-1}|_{1}^2+c_{0} |w_{\perp}|_{1}^2
+ | w |_{S(ij)}^2 \big)
\leq 
  C |H|_{2}\Big( |w|_{1} + C(j) |w|_{S(ij)}\Big).
\end{equation}
 By using the inequality
\begin{equation}\label{young}
ab\leq \eps a^2+ \frac{1}{4\eps}\,b^2,\quad \forall\, \eps>0,\quad \forall\, (a,b)\in\mathbb{R}^2,
\end{equation}
with $\eps$ small enough, we can incorporate $| w|_{S(ij)}$ in
 the left hand side of \eqref{young0} and arrive at
\beq\label{wperp}
|w_{\perp}|_{1}^2  +   | w |_{S(ij)}^2\leq C\big( |\alpha |^2 + 
|H|^2_{2} +  |H|_{2}\, |w|_{1}\big).
\eeq
In what follows $C$ is a large number which may change from line to line and depends on $\gamma_{0}$  
and $K$ but not on $\tau$. The next step is to estimate $\alpha$ and $w_{0}$.
We use the
decomposition \eqref{dec} and take the scalar product of
\eqref{wL} with $ \alpha\varphi_{-1}$ and with $w_{0}$ respectively  to get
\begin{eqnarray*}
&& (\gamma_{0}+ i \tau ) |\alpha|^2   = - \alpha \big( (w, L J (ij)(\varphi_{-1})) + (J(ij)S(ij)  
w , \varphi_{-1}) + ( J(ij) H , \varphi_{-1})\big) \\
&&   (\gamma_{0}+ i \tau )|w_{0}|^2  = -  (w, L J(ij)  w_{0}) + (J(ij)S(ij) w, w_{0})
 + ( J (ij) H, w_{0})
\end{eqnarray*}
and hence, we can take  the  modulus  and     add the two identities to get thanks to \eqref{w0} and \eqref{JSw} that 
$$ 
(\gamma_{0} + |\tau| ) (  |\alpha |^2 + | w_{0} |^2 )  \leq C \big( |\alpha |^2 + | w_{0}|^2 + 
|w_{\perp}|^2  + |w|_{S(ij)}^2   + |  H |^2  \big)
$$
which we can rewrite as
\beq
\label{ab}
(\gamma_{0} + |\tau| -C  ) (  |\alpha |^2 + | w_{0} |^2 )  \leq 
C \big(  |w_{\perp}|^2  + |w|_{S(ij)}^2   + |  H |^2  \big).
\eeq
Combining \eqref{wperp} and \eqref{ab}, we infer that there exists $M>0$  such
that for $| \tau | \geq M$,
\beq\label{wL2}
|w|_{1}^2 +  | w |_{S(ij)}^2 \leq C\big(  |H|_{2}\, |w|_{1} + |H|_{2}^2\big).
\eeq 
To conclude, we use again the inequality (\ref{young}) and we obtain
\beq\label{H1HF}
|w|_{1}^2 +  | w |_{S(ij)}^2 \leq C |H|_{2}^2.
\eeq
This proves \eqref{estHF} for $s=1$. Note that moreover \eqref{H1HF} gives a
control of   $|w|_{S(ij)}^2$ which is interesting when $j \neq 0$.
   
In order to estimate higher order derivatives,  we use the operator $M_{s}$ defined in section \ref{Ms}. 
By taking the scalar of  product of \eqref{wL}  by $M_{s}u$ and taking the real part,  
since $M_{s}$  is self-adjoint, we find
$$
\gamma_{0}(w, M_{s}w) \leq \Re \big( (J(L + S(ij)) w, M_{s} w) + (JH, M_{s}w) \big)
$$
and hence we find thanks to \eqref{Ms1}, \eqref{Ms2},
$$ 
\gamma_{0} \big( |w|_{s}^2 - C |w|^2_{s-1} \big) \leq  C |w|_{s}\, |w|_{s-1} +C |JH|_{s} |w |_{s}.
$$
This yields  by a new use of the Young inequality (\ref{young})
$$ 
|w|_{s}^2 \leq C\big( |w|_{s-1}^2 +  |JH|_{s}^2 \big)
$$
and hence, thanks to the assumption on $J$, we have
$$ |w|_{s}^2 \leq C\big( |w|_{s-1}^2 +  |H|_{s+1}^2 \big). $$
An induction argument completes the proof of Lemma~\ref{HF}.
\subsubsection{Proof of Lemma \ref{BF}}
The assertion of this lemma is a part of our assumptions. Indeed, for
$\sigma= \gamma_{0}+ i\tau$ , $|\tau|\leq M$ and every
 $j$, $|j| \leq k$, we have by choice of $\gamma_{0}$
  that there is no unstable modes on this line which
   is equivalent to $D(\sigma, j) \neq0$.
    Consequently,  the assumption \eqref{resper} gives the result.

The proof of Proposition~\ref{propuk} is therefore also completed.

\subsection{Nonlinear instability  (end of the proof of Theorem~\ref{theoper})} 
We look for a solution of (\ref{ham2}) in the form $u=Q+u^{ap}+w$. Then the
problem for $w$ to be solved  is
$$
\partial_{t}w  = 
{\mathcal J}(\partial_y) ( L_{0}w+\nabla F(Q+u^{ap}+w)-\nabla F(Q+u_{ap})
+\mathcal{S}(\partial_{y}) w )-{\mathcal J}(\partial_y)G
$$
with zero initial data, where thanks to Proposition~\ref{propuk},
\begin{equation}\label{raul1}
\|{\mathcal J}(\partial_y)G(t,\cdot)\|_{s}\leq C_{M,s}\delta^{M+2}e^{(M+2)\Re(\sigma_0) t},
\end{equation}
as far as $0\leq t\leq T^{\delta}$, where
$$
T^\delta\equiv \frac{\log(\kappa/\delta)}{\Re(\sigma_0)}
$$
with $\kappa\in ]0,1[$ small enough, the smallness restriction on $\kappa$ to be fixed in this section.
Thanks to our nonlinear assumption \ref{asnon} and the structure of $u^{ap}$, $w$ is defined for small times.
Next, since $\mathcal{J}(\partial_y)L_{0}$ and $\mathcal{J}(\partial_y)\mathcal{S}$ are skew symmetric, $w$ enjoys the energy estimate
 $$ \frac{d}{ dt }  \| w(t) \|_{s}^2 
  \leq  \sum_{ | \alpha |\leq s}
   \int_{0}^1  \Big(\!\!\Big(  \partial^\alpha \big(  \mathcal{J}D\nabla F (Q+ u^{ap}(t) + \sigma w(t))
    \cdot w(t)  \big) , \partial^\alpha w(t) \Big)\!\!\Big)\, d\sigma
     + \| \mathcal{J} G (t)\|_{s} \, \| w(t)\|_{s}$$
 where $(\!(\cdot, \cdot )\!)$ denotes the $L^2(\mathbb{R}\times \mathbb{T}_{a})$ scalar product.  
Let us  define a maximum time $T^*$ such that
$$ 
T^*= \sup \{ T\,:\,T\leq T^\delta,\,\, {\rm and}\,\,  \forall\, t \in [0,T], ||w(t)||_{s} \leq 1 \}.
$$
($T^*$ is well-defined since $w(0)=0$).
Consequently, we can use \eqref{tame}, with $s\geq s_0$ to get 
\begin{equation}\label{raul2}
\|w(t)\|_{s}^2 
\leq
\int_{0}^t\big(\|{\mathcal J}G(\tau) \|_{s} \|w(\tau)\|_{s}  + 
\omega(C+\kappa C_{M,s})\|w(\tau)\|_{s}^2 \big) \, d\tau, 
\end{equation}
provided also that $t\leq T^\star$. 
Combining (\ref{raul1}) and (\ref{raul2}), we obtain that for $t\in [0,T^*[$,
$$
\|w(t)\|_{s}^2 \leq 
\omega(C+\kappa C_{M,s})\int_{0}^{T}\|w(\tau)\|_{s}^2d\tau + C_{M,s}\delta^{2(M+2)}e^{2(M+2)\Re(\sigma_0) t}\,.
$$
We take an integer $M$ large enough so that $2(M+2)\Re(\sigma_0)-\omega(C)\geq 2$. At this place
we fix the value of $M$. We then choose $\kappa$ small enough so that
$$
2(M+2)\Re(\sigma_0)-1>\omega(C+\kappa C_{M,s})-\omega(C)\,.
$$
Such a choice of $\kappa$ is indeed possible thanks to the continuity
assumption on $\omega$. By a bootstrap argument and the Gronwall lemma, we infer that $w(t)$ is
defined for $t\in [0,T^\delta]$ and that 
$$
\sup_{0\leq t\leq T^\delta}\|w(t,\cdot)\|_{s}\leq C_{M,s}\kappa^{M+2}\,.
$$
In particular
\begin{equation}\label{reste}
\|w(T^\delta,\cdot)\|_{L^2(\rit\times \tit_{L})}\leq C_{M,s}\kappa^{M+2}\,.
\end{equation}
Let us denote by $\Pi$ the projection on the nonzero modes in $y$.
For an arbitrary $w\in \mathcal{F}$ (an $L^2(\R)$ function depending only on $x$) one has $\Pi(w)=0$. 
On the other hand the first term of $u^{ap}$ satisfies $\Pi(u^{0})=u^{0}$ and 
therefore using (\ref{estuk}) 
\begin{eqnarray*}
\|\Pi(u^{ap}(t,\cdot))\|_{L^2}
& \geq &
c_s\delta e^{\Re(\sigma_0) t}-\sum_{k=1}^{M}\delta^{k+1}\|\Pi(u^{k})\|_{L^2}
\\
& \geq &
c_s\delta e^{\Re(\sigma_0) t}-\sum_{k=1}^{M}C_{k,s}\delta^{k+1}e^{(k+1)\Re(\sigma_0) t}\,.
\end{eqnarray*}
Therefore for $\kappa$ small enough one has 
\begin{equation}\label{parvo}
\|\Pi(u^{ap}(T^\delta,\cdot))\|_{L^2(\rit\times \tit_{L})}
\geq
\frac{c_s \kappa}{2}\, .
\end{equation}
Using (\ref{reste}) and (\ref{parvo}), we may write that for every $w\in \mathcal{F}$,
\begin{eqnarray*}
\|u^{\delta}(T^\delta,\cdot)-w\|_{L^2}
& \geq & 
\|\Pi(u^{\delta}(T^\delta,\cdot)-w)\|_{L^2}
\\
& = &
\|\Pi(u^{\delta}(T^\delta,\cdot)-Q(\cdot))\|_{L^2}
\\
& = &
\|\Pi(u^{ap}(T^\delta,\cdot)+w(T^\delta,\cdot))\|_{L^2}
\\
& \geq & \frac{c_s \kappa}{2}-\|\Pi(w(T^\delta,\cdot))\|_{L^2}
\\
& \geq & 
\frac{c_s \kappa}{2}-\|w(T^\delta,\cdot)\|_{L^2}
\\
& \geq &
\frac{c_s \kappa}{2}-C_{M,s}\kappa^{M+2}\,.
\end{eqnarray*}
A final restriction on $\kappa$ may insure that the right hand-side of the last inequality 
is bounded from below by a fixed positive constant $\eta$ depending only on $s$
(in particular $\eta$ is independent of $\delta$).
This completes the proof of Theorem \ref{theoper}.
\qed
\section {Proof of Theorem \ref{theoloc} (localized perturbations)}
The first step is to  find the most unstable eigenmode which solves
\beq\label{eqvp2}\sigma U = J(ik) ( LU + S(ik) U ).\eeq
Since now $k$ is in $\mathbb{R}$ it is slightly more complicated. We begin with a few preliminary remarks which
allow to reduce the search for unstable eigenmodes in a compact set. 
Thanks to Lemma~\ref{krou} and \eqref{klarge}, we already now that unstable
eigenmodes must be seek only for $k \in [0,K]$, where $K$ is fixed by
assumption (\ref{klarge}). Moreover,  by   taking the scalar product  of
\eqref{eqvp2} by $U $,   
and then taking the real part, we  get
\begin{equation*}
\Re(\sigma) |U|^2 \leq \mbox{Re } (J(ik)R  U,  U) 
\end{equation*}
since  $J(ik)L_{0}$ and $J(ik)S(ik)$ are skew symmetric. Since $J(ik)$ is also skew symmetric, we also have
  $$  \mbox{Re } (J(ik)R  U,  U) =  ([J(ik) , R]U, U)$$
  and hence, thanks to \eqref{JR}, we obtain
 $$ \Re(\sigma) |U|^2 \leq C |U|^2.$$
Consequently, there is no  nontrivial  solution for $\Re(\sigma)$ sufficiently large. Next,
by using the result of Lemma~\ref{HF} for $H=0$, we find that for every $\gamma_{0}>0$,
there exists $C(\gamma_{0}, K)$ such that there is no nontrivial  unstable
solution of \eqref{eqvp2} for $ \Re\,  \sigma \geq \gamma_{0}$ and
$|\Im\,\sigma |\geq C(\gamma_{0}, K)$. 
        
Now, let us assume that  the  unstable eigenmode given by our assumption
is such that $\Re\, \sigma = \delta$. Then thanks to the previous remarks, 
the most unstable eigenmode has to be seek in the compact set
$$
(\sigma, k) \in  \mathcal{R}\equiv\big\{ \delta/2 \leq Re\, \sigma 
\leq C, \,\, |Im\, \sigma | \leq C(\delta, K) |,\,\, |k |\leq K\big\}.
$$
Moreover,   we have an unstable eigenmode if and only
if  $(\sigma,k)$ is a zero of the corresponding extended  Evans function $
\tilde{D}(\sigma, k)$
which is  an analytic function in $\{\Re\, \sigma >0 \}\times \mathbb{R}$.
  We have already proven that for each $k$ there is at most one zero
with $\Re\, \sigma >0$. By Rouche Theorem,  if there exists
$(\sigma_{0}, k_{0})$ with $\Re\, \sigma_{0}>0$ such that $\tilde{D}(\sigma_{0}, k_{0})
= 0$, there exists a vicinity of $\sigma_{0}$ and $k_{0}$ such that for each
$k$, there is exactly one  zero $\sigma = \sigma(k)$ of $\tilde{D}$. Moreover, $k
\rightarrow \sigma(k)$ is analytic. Indeed, we have the explicit expression 
$$ 
\sigma(k) = C \int_{\Gamma}  z \frac{\partial_{\sigma } \tilde{D}(z, k)}{ 
\tilde{D}(z, k) }\, dz,
$$
where $\Gamma$ is a circle which contains $\sigma_{0}$ and $C$ is a constant and hence the analycity of $\tilde{D}$ gives
the analycity of $\sigma$.  If we define $\Omega = \big\{k, \, \exists\,\sigma, \, \Re\, \sigma>\delta/2,\, 
\tilde{D}(\sigma, k)= 0  \big\}$, this proves in particular that $\Omega$ is an open
 bounded  set
(and non empty thanks to the assumption of the existence of an unstable mode)  of 
$\mathbb{R}$. One can decompose $\Omega$ as $\Omega= \cup_{m} I_{m}$ where
$I_{m}$ are disjoint, open and bounded  
intervals which are the connected components of $\Omega$. On each $I_{m}$ the above considerations
prove that there exists an analytic function $k \rightarrow \sigma(k)$ such
that $\sigma(k)$ is the only zero of $\tilde{D}$ in $\Re\, \sigma >0.$ 
We shall prove next that $k \rightarrow \Re\, \sigma(k)$ has a continuous
extension to $\overline{I_{m}}$. Indeed, if  $k_{n}$ is a sequence converging to an extremity $\kappa$ of  $I_{m}$, 
since $\sigma(k_n)$ is bounded ($\sigma (k_{n}) \in \mathcal{R}$), then we can
extract a sub-sequence not relabelled such that $\sigma(k_{n})$ tends to some $\sigma$. Moreover, we also 
have $\Re\, \sigma \geq \delta /2$,  and $\tilde{D}(\sigma, \kappa)=0$, so $\sigma$ is
the only unstable zero of $\tilde{D}(\cdot, \kappa)$. This allows to get that 
$\lim_{k \rightarrow \kappa, k \in I_{m}} \sigma(k)= \sigma$ and hence to define a continuous function on $\overline{I_{m}}$.
Finally, we also notice that if $\partial I_{m} \cap \partial I_{m'}
\neq\emptyset$, then the continuations must coincide again thanks to the fact
that there is at most one unstable eigenmode. Consequently, we have actually a
well-defined continuous function $k \rightarrow \sigma(k)$ on  $\overline{\Omega}$
which is a compact set. This allows to define the most unstable eigenmode as 
$$ 
\sigma_{0} = \Re\, \sigma(k_{0})= \sup\{\sigma_{0}(k)=\Re\, \sigma(k), \quad k
\in \overline{\Omega}\}.
$$ 
Note that, since we have assumed that there exists  an unstable mode, $\sigma_{0}$ is
positive and also that $k_{0}\neq 0$ thanks to the assumption \eqref{evans1Dt}.
Moreover, $\sigma_{0}(k)$ is an analytic function in the vicinity of $ k_{0}$  and hence,  there exists $m\geq 2$ 
\begin{equation}\label{nondeg}
[\Re(\sigma)]'(k_0)=\cdots=[\Re(\sigma)]^{(m-1)}(k_0)=0,\quad [\Re(\sigma)]^{(m)}(k_0)\neq 0.
\end{equation}
Let $I$ be an interval containing $k_0$ which does meet zero.
For $k\in I$, let us denote by $U(k)$ the unstable mode corresponding to transverse
frequency $k$ and amplification parameter $\sigma(k)$. Then we set
$$
\Phi(t,x,y,k,\sigma(k))=2\Re\big(e^{\sigma(k) t } e^{ i k y }\,U(k)\big),
$$
where the dependence of $x$ of $\Phi$ is in $U(k)$. Further we define
$$
u^0(t,x,y)\equiv \int_{I}\Phi(t,x,y,k,\sigma(k))dk\,.
$$
The function $u^0$ is the first term of our approximate solution, i.e. it is a
solution of
$$
(\partial_t-{\mathcal J}(\partial_y)L-{\mathcal J}(\partial_y){\mathcal S}(\partial_{y}))u^0=0\,.
$$
Recall that $\sigma_0\equiv \Re(\sigma)(k_0)$.
Thanks (\ref{nondeg}), we can apply the Laplace method (see e.g. \cite{JD,F}) 
and obtain that  for every $s\geq 0$ there exists
$c_{s}\geq 1 $ such that for every $t\geq 0$
\begin{equation}\label{fed1}
\frac{1}{c_{s}}\frac{1}{(1+t)^{\frac{1}{2m}}}e^{\sigma_0 t}
\leq\|u^0(t,\cdot)\|_{H^{s}(\mathbb{R}^2)}\leq \frac{c_{s}}{(1+t)^{\frac{1}{2m}}}e^{\sigma_0 t}\,.
\end{equation}
As in the previous section, we look for an approximate solution of the form
\begin{equation}\label{u_app_bis}
u^{ap}= \delta\Bigl( u^0  + \sum_{ k=1}^M \delta^k u^k \Bigr), \quad u^k \in L^2(\mathbb{R}^2)\,,
\end{equation}
where $\delta\ll 1$ and $M\gg 1$ and $u^k$, $ 1 \leq k \leq M$ are solutions
of (\ref{uk}) with zero initial data.
Observe that the Fourier transform of $u^k$ with respect to $y$ is  compactly supported.
Thus using  the Fourier transform in $y$, the Laplace transform in $t$, 
and  (\ref{fed1}),   we  can deduce as in the proof of Theorem
 \ref{theoper}  the bounds
\begin{equation}\label{fed2}
\|u^k(t,\cdot)\|_{s}\leq \frac{c_{s,k}}{(1+t)^{\frac{k+1}{2m}}}e^{(k+1)\sigma_0 t}
\end{equation}
from the following resolvent estimate.
\begin{proposition}\label{resr}
Consider $w(\tau)$ the solution of
\begin{equation*}
 (\gamma_{0}+ i \tau) w= J(ik)(L+S(ik))w + J(ik)H,\quad \sigma_0<\gamma_{0}<2\sigma_0.
\end{equation*}
Then there exists $q\geq 0$ such that for every integer $s\geq 1$, and every $K>0$,  there exists 
$C(s, \gamma_{0}, K)$ such that  for every  $k \in \mathbb{R}\backslash\{ 0\} $,
$|k | \leq K$, and $\tau \in \mathbb{R}$, we have the estimate
\beq\label{estkr}
|w(\tau)|_{s}^2 \leq C(s, \gamma_{0}, K ) |H(\tau)|_{s+q}^2. 
\eeq
\end{proposition} 
\begin{proof}[Proof of Proposition~\ref{resr}.]
As in the proof of Proposition~\ref{resolvant}, we can split the proof of
\eqref{estkr} into large $|\tau|$ estimates and bounded $|\tau|$ estimates. The
proof of  the large $|\tau|$ estimate  was already proved in  Lemma~\ref{HF}.
As already noticed, there is no difference between the continuous and
discrete cases in $k$ for this estimate.  
To treat the small $|\tau|$ case, we use  the assumptions of Section~\ref{hyploc}.
 By choice of $\gamma_{0}$ and thanks to the assumption \eqref{evans1Dt},
  we can use \eqref{resloc} for $\mathcal{K}= \{\gamma_{0}+ i \tau, \, 
   |\tau |\leq M\}.$
\end{proof}
With (\ref{fed1}) and (\ref{fed2}) at our disposal, we may complete the
instability proof as in the previous section. We choose $T^{\delta}>0$ such that
$$
\frac{e^{T^{\delta}\,\sigma_0}}{[1+T^{\delta}]^{\frac{1}{2m}}}=\frac{\kappa}{\delta},
$$
where $\kappa>0$ is small enough to be fixed.
Again, we write the solution of (\ref{ham2}) in the form $u=Q+u^{ap}+w$ with
$w$ solution of 
$$
\partial_{t}w  = 
{\mathcal J}(\partial_y) ( L_{0}w+\nabla F(Q+u^{ap}+w)-\nabla F(Q+u_{ap})
+\mathcal{S}(\partial_{y}) w )-{\mathcal J}(\partial_y)G
$$
with zero initial data.
Thanks to (\ref{fed2}), we have that
\begin{equation}\label{raul1_bis}
\|{\mathcal J}G(t,\cdot)\|_{s}\leq
C_{M,s}\delta^{M+2}\frac{e^{(M+2)\Re(\sigma_0)  t}}{
(1+t)^{\frac{M+2}{2m}}
},\quad 0\leq t\leq T^{\delta}\,.
\end{equation}
Then thanks to our assumptions, $w$ enjoys the energy estimate
\begin{equation}\label{raul2_bis}
\|w(t)\|_{s}^2 
\leq
\int_{0}^t\big(\|{\mathcal J}G(\tau) \|_{s} \|w(\tau)\|_{s}  + 
\omega(C+\kappa C_{M,s})\|w(\tau)\|_{s}^2 \big) \, d\tau, 
\end{equation}
provided also that $t\leq T^\delta$ and $t$ small. Let us  define $T^*$ by
$$ 
T^*\equiv \sup \{ T\,:\,T\leq T^\delta,\,\, {\rm and}\,\,  \forall\, t \in [0,T], \|w(t)\|_{s} \leq 1 \}.
$$
($T^*$ is well-defined since $w(0)=0$).
Thanks to (\ref{raul1_bis}) and (\ref{raul2_bis}), we obtain that for $t\in [0,T^*[$,
$$
\|w(t)\|_{s}^2 \leq 
\omega(C+\kappa C_{M,s})\int_{0}^{T}\|w(\tau)\|_{s}^2d\tau +
C_{M,s}\delta^{2(M+2)}
\frac{e^{2(M+2)\Re(\sigma_0)  t}}{
(1+t)^{\frac{M+2}{m}}}\,.
$$
We fix an integer $M$ large enough so that $2(M+2)\Re(\sigma_0)-\omega(C)\geq 2$. 
We then choose $\kappa$ small enough so that
$$
2(M+2)\Re(\sigma_0)-1>\omega(C+\kappa C_{M,s})-\omega(C)\,.
$$
Using the inequality
$$
\int_{0}^{t}
\frac{e^{2(M+2)\sigma_0 \tau-\omega(C+\kappa C_{M,s}) \tau}}{(1+\tau)^{\frac{M+2}{2m}}}d\tau
\leq \frac{\tilde{C}\,e^{2(M+2)\sigma_0 t-\omega(C+\kappa C_{M,s}) t}}{(1+t)^{\frac{M+2}{2m}}}\,.
$$ 
a bootstrap argument and the Gronwall lemma, we infer that $w(t)$ is
defined for $t\in [0,T^\delta]$ and that 
$$
\|w(T^\delta,\cdot)\|_{s}\leq C_{M,s}\kappa^{M+2}\,.
$$
In particular
\begin{equation}\label{reste_bis}
\|w(T^\delta,\cdot)\|_{L^2(\rit\times \tit_{L})}\leq C_{M,s}\kappa^{M+2}\,.
\end{equation}
Let $I_0$ be a neighborhood of zero which does not meet $I$. Let us fix $\varphi \in
C^{\infty}_0(\R)$ vanishing on $I_0$ and equal to one on $I$. 
Let us denote by $\Pi $ the map acting on ${\mathcal S}'(\R^2)$, defined via
the Fourier transform as
$$
\widehat{\Pi(u)}(\xi_1,\xi_2)=\varphi(\xi_2)\hat{u}(\xi_1,\xi_2).
$$
The  Fourier multiplier $\Pi$ is bounded on $L^2(\R^2)$ and we notice that the first term of $u^{ap}$ satisfies $\Pi(u^{0})=u^{0}$.  
Therefore, we have that 
\begin{eqnarray*}
\|\Pi(u^{ap}(t,\cdot))\|_{L^2}
& \geq &
c_s\delta \frac{e^{\Re(\sigma_0) t}}{(1+t)^{\frac{1}{2m}}}-\sum_{k=1}^{M}\delta^{k+1}\|\Pi(u^{k})\|_{L^2}
\\
& \geq &
c_s\delta \frac{e^{\Re(\sigma_0) t}}{(1+t)^{\frac{1}{2m}}}
-\sum_{k=1}^{M}C_{k,s}\delta^{k+1}
\frac{e^{(k+1)\Re(\sigma_0) t}}{(1+t)^{\frac{k+1}{2m}}}\,.
\end{eqnarray*}
Therefore for $\kappa$ small enough one has 
\begin{equation}\label{parvo_bis}
\|\Pi(u^{ap}(T^\delta,\cdot))\|_{L^2(\rit\times \tit_{L})}
\geq
\frac{c_s \kappa}{2}\, .
\end{equation}
Finally, since for $w\in \mathcal{F}\subset{\mathcal S}'(\R^2)$, 
 where $\mathcal{F}$  which is defined in statement of Theorem \ref{theoloc}
  is the set of functions (or  tempered distributions) which depends
   only on $x$,  
   we have that $\Pi(w)=0$, by 
using (\ref{reste_bis}), (\ref{parvo_bis}), we can write that for every $w\in \mathcal{F}$,
\begin{eqnarray*}
\|u^{\delta}(T^\delta,\cdot)-w\|_{L^2}
& \geq & 
\|\Pi(u^{\delta}(T^\delta,\cdot)-w)\|_{L^2}
\\
& = &
\|\Pi(u^{\delta}(T^\delta,\cdot)-Q(\cdot))\|_{L^2}
\\
& = &
\|\Pi(u^{ap}(T^\delta,\cdot)+w(T^\delta,\cdot))\|_{L^2}
\\
& \geq & \frac{c_s \kappa}{2}-\|\Pi(w(T^\delta,\cdot))\|_{L^2}
\\
& \geq & 
\frac{c_s \kappa}{2}-\|w(T^\delta,\cdot)\|_{L^2}
\\
& \geq &
\frac{c_s \kappa}{2}-C_{M,s}\kappa^{M+2}\,.
\end{eqnarray*}
A final restriction on $\kappa$ may insure that the right hand-side of the last inequality 
is bounded from below by a fixed positive constant $\eta$.
This completes the proof of Theorem \ref{theoloc}.
\qed

\section{Examples}
In this section we give a number of examples when our general result of
Theorem~\ref{theoper} and Theorem~\ref{theoloc} applies with an unstable mode
generated by Lemma~\ref{eigen}.
\subsection{The generalized KP-I equation}
The $1d$ model is the gKdV equation
\begin{equation}\label{gKdV}
u_{t}=\partial_{x}(-\partial_{x}^{2}-u^p),\quad p=2,3,4,\quad u:\R\rightarrow \R.
\end{equation}
For simplicity, we consider only the case of power nonlinearities but more
general nonlinear interactions may be considered too.
Equation (\ref{gKdV}) has a solution of the form $u(t,x)=Q(x-t)$ with $Q$ 
smooth with exponential decay. We even have an explicit formula for $Q$ namely
\beq
\label{QKP}
Q(x) = \Bigl( \frac{ p+ 1 } { 2 } \Bigr)^{ \frac{1}{p-1} } \,
\Bigl(   \mbox{sech}^2({ \frac{ (p-1) x}{2}})\Bigr)^{\frac{1}{p-1} }.
\eeq
The solution $u(t,x)=Q(x-t)$ describes the displacement of the profile $Q$ from left to the right with speed
one. One also has the solution
\begin{equation}\label{speed}
u_{c}(t,x) =  c^{\frac{1}{p-1} } \,Q(  \sqrt{c}(x- c t )), \quad c>0
\end{equation}
which correspond to a solitary wave with a positive speed $c$.
We  restrict our considerations only to speed one solitary
waves since the case  of arbitrary speeds can be reduced to speed one
 by a change of scale because of \eqref{speed}.

By changing $x$ into $x-t$, we observe  that $Q$ is stationary solution of 
\begin{equation}\label{gKdVbis}
u_{t}=\partial_{x}(-\partial_{x}^{2}u+u-u^p)
\end{equation}
which fits into the framework
  of  section \ref{s1D} with $d=1$, 
  $$ J=\partial_x, \quad L_{0}=-\partial_{x}^{2}+{\rm Id},
\quad  F(u)=-\frac{u^{p+1}}{p+1}.$$
   Obviously, the first assumptions of  section \ref{s1D} are matched.
    Moreover, we have, 
$$
L=-\partial_{x}^{2}+{\rm Id}-pQ^{p-1}{\rm Id},\quad  R = - p Q^{p-1}{\rm Id }\,.
$$
The spectral condition
 \eqref{spectre}  on $L$ is satisfied by Sturm-Liouville theory since $Q'$ has only one zero (see \cite{Ben,BSS}). 

 The transversally perturbed model is the gKPI equation which reads
  in the moving frame 
\begin{equation}\label{gKPI}
u_{t}=\partial_{x}(-\partial_{x}^{2}u+u-u^p+\partial_{x}^{-2}\partial_{y}^{2}u).
\end{equation}
Consequently, we have  $J(ik)= \partial_{x}$,  $\mathcal{S}(\partial_{y})=\partial_{x}^{-2}\partial_{y}^{2}$ and hence 
$S(ik)=-k^2\partial_{x}^{-2}$.

The assumptions of section \ref{sJ} are obviously met.  In particular, since
$$ [R, J] w =  p ( Q^{p-1})_{x} w,$$
 \eqref{JR} is true.
 
 Next, one can  also easily check the assumption of section \ref{sS}.
 Note that 
$ |w|_{S(ik)}^2 \equiv k^2|\partial_{x}^{-1}w|^2,$
  hence, assumption (\ref{JSw}) is  satisfied with $C(k)=|k|$.
 
Let us next check  the assumption (\ref{klarge}).
We have
$$
(Lv, v) + (S(ik)v, v )\geq |v_{x}|^2 +  |v|^2  + k^2 |\partial_{x}^{-1}v|^2 - C |v|^2,
$$
where $C= p\|Q^{p-1}\|_{L^\infty(\R)}$ and hence using the Fourier transform, we find
$$ 
(Lv, v) + (S(ik)v, v )\geq (2\pi)^{-1} 
\int_{\R} \Big(\xi^2 + \frac{k^2}{\xi^2} + 1 - C\Big) |\hat{v}(\xi) |^2 \, d\xi.
$$
and since $ \xi^2 + k^2 /\xi^2 \geq 2C$ for $|k| \geq C$, we get \eqref{klarge}
(more precisely for $|k|\geq C$, $ \xi^2 + k^2 /\xi^2 \geq 2|k|\geq 2C$).

Let us next turn to the assumption on the  eigenvalue problem in the
context of (\ref{gKPI}).
The resolvent equation reads  
\begin{equation}\label{matritza_bis}
\sigma u=\partial_{x}(-\partial_x^{2}+1-pQ^{p-1})u-k^2\partial_{x}^{-1}u+F_{x}.
\end{equation}
To prove the existence of the Evans function and \eqref{resper}, \eqref{resloc},
 we shall use the criterion of section \ref{criteria}.
 Let us define
 $$ R(\sigma, k) = \left\{ \begin{array}{ll} \partial_{x}, \quad \mbox{ if } k\neq 0, \\
   {\rm Id},  \quad  \mbox{ if } k= 0. \end{array}\right.$$
   then we directly  find that
    $ R(\sigma - J( L+ S(ik))=P_{1}(\sigma, k)$ is a differential operator
     of order $4$ for $k \neq 0$ and $3$ for $k=0$.
      Consequently, the assumption \ref{factorisation} is matched
       with an empty second block.
       
 For $k \neq 0$, we have \eqref{edo} with 
$$
A(x,\sigma,k)=
\left( 
\begin{array}{cccc}
0 & 1 & 0 & 0 \\ 
0 & 0 & 1 & 0 \\
0 & 0 & 0 & 1 \\
-k^2-p\partial_x^{2}(Q^{p-1}) & -\sigma-2p\partial_x(Q^{p-1}) & 1-pQ^{p-1} & 0
\end{array} 
\right).
$$
Thus
$$
A_{\infty}(\sigma,k)=
\left( 
\begin{array}{cccc}
0 & 1 & 0 & 0 \\ 
0 & 0 & 1 & 0 \\
0 & 0 & 0 & 1 \\
-k^2 & -\sigma & 1 & 0
\end{array} 
\right).
$$
The eigenvalues  of $A_{\infty}(\sigma,k)$  are the roots of the  polynomial $P$ 
\beq\label{P}
P(\lambda ) = \lambda^4 - \lambda^2 + \sigma \lambda + k^2
\eeq
and hence are not purely imaginary when $\Re\, \sigma >0$, $k \neq0$.
 Moreover, there are two of positive real part and two of negative real part.
For $k=0$, we have 
$$
A(x,\sigma, 0)=
\left( 
\begin{array}{cccc}
0 & 1 & 0  \\ 
0 & 0 & 1   \\
-\sigma-p\partial_x(Q^{p-1}) & 1-pQ^{p-1} & 0
\end{array} 
\right)
$$
and thus
$$
A_{\infty}(\sigma, 0)=
\left( 
\begin{array}{cccc}
0 & 1 & 0 \\ 
0 & 0 & 1  \\
-\sigma & 1 & 0 
\end{array} 
\right).
$$
The characteristic polynomial of $A_{\infty}(\sigma, 0)$ is
$p(\lambda)=-\lambda^3+\lambda+\sigma$ and thus for $\Re(\sigma)>0$ the eigenvalues of $A_{\infty}(\sigma, 0)$ do not meet the
imaginary axis. 

Consequently, the  existence of the Evans function follows from Lemma
  \ref{existevans}.

 Finally, since the KdV solitary wave
is stable (see e.g. \cite{PW}), 
we have $D(\sigma, 0) \neq 0$ when $\Re\, \sigma >0$ and hence the assumption 
 \eqref{evans1D} is met. Consequently, \eqref{resper} follows from \eqref{lemper}

To handle the localized case, we note that  
when $k$  tends to zero,  there is a single root $\lambda=0$ of (\ref{P}) on the imaginary axis
 and hence, there is spectrum of $A_{\infty}(\sigma, 0^+)$ on the imaginary
  axis.
More precisely,  for $k\sim 0$ this root behaves  as
\beq
\label{muKP}
 \mu(\sigma,k) \sim -\frac{k^2}{ \sigma}.
\eeq
Consequently,  there  is only  one of the negative real part roots of (\ref{P}) which goes to
zero. Since $\mu(\sigma,k)$ is analytic, we can use the Gap lemma \cite{GZ}, \cite{KS}
to get the continuation of the Evans function. Moreover, for $k$ close to zero, we can write the Evans function as
$$ 
\tilde{D}(\sigma, k) = 
\left| \begin{array}{cccc}
\varphi_{1}^-(0,\sigma,k) & \varphi_{2}^-(0,\sigma,k) & \varphi_{1}^+(0,\sigma,k) & \varphi_{2}^+(0,\sigma,k) \\ 
\partial_{x}\varphi_{1}^-(0,\sigma,k) & \partial_{x}\varphi_{2}^-(0,\sigma,k) & \partial_{x}\varphi_{1}^+(0,\sigma,k) & \partial_{x}\varphi_{2}^+(0,\sigma,k) \\
\partial_{x}^2 \varphi_{1}^-(0,\sigma,k) &
\partial_{x}^2\varphi_{2}^-(0,\sigma,k) &
\partial_{x}^2\varphi_{1}^+(0,\sigma,k) &
\partial_{x}^2\varphi_{2}^+(0,\sigma,k) \\
\partial_{x}^3 \varphi_{1}^-(0,\sigma,k) &
\partial_{x}^3\varphi_{2}^-(0,\sigma,k) &
\partial_{x}^3\varphi_{1}^+(0,\sigma,k) &
\partial_{x}^3\varphi_{2}^+(0,\sigma,k)
\end{array} \right|,
$$
where $\varphi_{i}^\pm(x,\sigma,k)$, $i=1,2$ decay when $x$ goes to $\pm \infty$ for $k \neq
0$. When $k=0$,  $\varphi_{i}^-$, $i=1,2$ and 
$\varphi_{1}^+$ keep this  property as $\varphi_{2}^+(x,\sigma, 0)=c+{\mathcal
O}(e^{-\alpha|x|})$, where $c\neq 0$ and $\alpha>0$ are some fixed constants.
Note that $\varphi_{i}^-(x,\sigma, 0)$, $i=1,2$ and $\varphi_{1}^+(x,\sigma, 0)$ actually solve
\beq\label{kdv1}
\sigma u =\partial_{x}(-\partial_x^{2}+1-pQ^{p-1})u
\eeq
which is the linearized KdV equation about the solitary wave whereas after
integration, we get that $\varphi_{2}^+(x,\sigma, 0)$ solves
\beq\label{phi1}
\sigma u =\partial_{x}(-\partial_x^{2}+1-pQ^{p-1})u+c\sigma,
\eeq
where the source term $c\sigma$ is identified by looking at the value at
$\infty$ of
$$
\sigma \varphi_{2}^+(x,\sigma,
0)-\partial_{x}(-\partial_x^{2}+1-pQ^{p-1})\varphi_{2}^+(x,\sigma, 0).
$$
Consequently, using  \eqref{kdv1}, \eqref{phi1} we can write the forth derivatives of
$\varphi_{i}^\pm(0,\sigma,0)$, $i=1,2$ as the same linear combinations of
lower order derivatives with an additional term $c\sigma$ for $\varphi_{2}^+(0,\sigma, 0)$.
Therefore, we can perform an operation
on the line of the determinant which defines the Evans function, to get that
$$ 
\tilde{D}(\sigma, 0) = \left| \begin{array}{cccc}
\varphi_{1}^-(0,\sigma,0) & \varphi_{2}^-(0,\sigma,0) & \varphi_{1}^+(0,\sigma,0) & \varphi_{2}^+(0,\sigma,0) \\ 
\partial_{x}\varphi_{1}^-(0,\sigma,0) & \partial_{x}\varphi_{2}^-(0,\sigma,0) & \partial_{x}\varphi_{1}^+(0,\sigma,0) & \partial_{x}\varphi_{2}^+(0,\sigma,0) \\
\partial_{x}^2 \varphi_{1}^-(0,\sigma,0) & \partial_{x}^2\varphi_{2}^-(0,\sigma,0) & \partial_{x}^2\varphi_{1}^+(0,\sigma,0) & \partial_{x}^2\varphi_{2}^+(0,\sigma,0) \\
0 & 0 & 0 & c\sigma
\end{array} \right|.
$$
Consequently, we get that $|\tilde{D}(\sigma,0)|= |c\sigma D(\sigma, 0)|$ where
$D(\sigma, 0)$ is the Evans function 
associated to the linearized KdV equation about the solitary wave.  Again, 
since the KdV solitary wave
is stable (see e.g. \cite{PW}),  we also have $\tilde{D}(\sigma, 0)$ does
 not vanish for $\mbox{Re } \sigma >0$ and hence, \eqref{evans1Dt}
  is verified. 
  Finally, \eqref{compensation} is also met in view of \eqref{muKP} since
  $$ R(\sigma, k) J(ik) S(ik)=  \partial_{xx} (- k^2 \partial_{x}^{-2}) = -k^2.$$
 Therefore, \eqref{resloc}
  follows from Lemma \ref{lemloc}.

 The assumptions of section~\ref{Ms} on the existence of a multiplier $M_{s}$  are also matched. 
Indeed, we can use the criterion given by Lemma \ref{lemmult}. Let us   set
$$ 
K_s w = r_{s}(x) \,w
$$
where $r_{s}$ is a  smooth and real valued  function.
 A few computation give
\begin{eqnarray*}
E_{s} u &= & \frac{1}{2} \partial_{x} \big( ( p  Q^{p-1})_{x} \partial_{x} u \big)
  -\frac{s}{2} \big( - (p Q^{p-1})_{x} \partial_{xx} u 
   + \partial_{xx} ( -  (p Q^{p-1})_{x} u )  \big)
    -\frac{1}{2}  [- \partial_{x}^3 + \partial_{x}, r_{s}] u \\
    & = &  \Big(  \big( \frac{1}{2}  + s \big)  (p Q^{p-1})_{x} +\frac{3}{2} (r_{s})_{x} \Big)
     \partial_{xx} u + \tilde{E}_{s}u
\end{eqnarray*}
where $\tilde{E}_{s}$ is a first order differential operator.
 Consequently, with the choice
 $$ r_{s} = - \big( \frac{1+2s}{3} \big)  p Q^{p-1},$$
 the properties \eqref{Ms1}, \eqref{Ms2} are verified. Notice that a similar
 argument can be performed each time we deal with a scalar equation,
 i.e. $d=1$ in our general framework.

   The ``nonlinear'' assumptions in the context of (\ref{gKPI}) are also  met.
In the context of (\ref{gKPI}), \eqref{ham222} becomes
\begin{equation}\label{four}
\partial_{t}u=-u_{xxx}+u_{x}+\partial_{x}^{-1}u_{yy}-\partial_{x}
\big[(u^a+u)^p-(u^a)^p\big]+\partial_x G,\quad u(0)=0.
\end{equation}
To check \eqref{tame}, we have to estimate
$$ 
\int \partial^\alpha  \partial_{x} \big( (w+v)^{p-1}\, v\big)
 \, \partial^\alpha v \, dxdy$$
 with $\partial^\alpha= \partial_{x}^{\alpha_{1}} \partial_{y}^{\alpha_{2}}, $
$  |\alpha_{1}| + |\alpha_{2}| \leq  s$. 
Therefore we need to study
$$
\int \partial^\alpha  \partial_{x} \big( w^q\, v^{r}\big)\, 
\partial^\alpha v \, dxdy,\quad 0\leq q\leq p-1,\quad q+r=p\,,
$$
where $w$ may only be putted in $L^\infty$ (or some of its derivatives).
If at least one of the derivatives of $\partial^\alpha  \partial_{x}$ acts on
$w^q$ then we can use the Gagliardo-Nirenberg-Moser estimates to get the needed bound.
Therefore it remains to study
$$
\int 
w^q\partial^\alpha  \partial_{x} \big(v^{r}\big)\, 
\partial^\alpha v \, dxdy=
r\,\int 
w^q\partial^\alpha   \big(v^{r-1}\partial_{x}v\big)\, 
\partial^\alpha v \, dxdy
$$
We write 
\begin{eqnarray*}
\int w^q\partial^\alpha   \big(v^{r-1}\partial_{x}v\big) \, \partial^\alpha v 
& = & 
\int w^q\, v^{r-1}\partial_{x}\partial^{\alpha}v \, \partial^\alpha v 
+
\int w^q [\partial^\alpha, v^{r-1}]\partial_{x}v\, \partial^\alpha v 
\\
& = & 
-\frac{1}{2}  \int \partial_{x}( w^q v^{r-1}  ) | \partial^\alpha v |^2\, dx dy 
+ 
\int w^q [\partial^\alpha, v^{r-1}]\partial_{x}v\, \partial^\alpha v 
\end{eqnarray*}
and hence \eqref{tame} follows by the Sobolev embedding and the classical tame commutator estimate
$$ 
\| [\partial^\beta , f]  g \| \leq C_{\beta} \Big( \|f\|_{k} \, \| g \|_{L^\infty}
 + \| \nabla f \|_{L^\infty} \, \| g \|_{k-1} \Big), \quad 1 \leq |\beta | \leq k.
$$ 

As already used in the general framework the estimate \eqref{tame} 
and the fact that $JL_{0}$ and $J\mathcal{S}$ are skew-symmetric  
allow to get and $\mathbb{H}^s$ energy estimate for \eqref{four}.


To get the well-posedness of \eqref{four}, 
the procedure is very classical and there are several
possibilities to achieve this conclusion. One possibility is to consider a
regularized version of (\ref{four}), for example
\begin{equation}\label{four_epsilon}
\partial_{t}u^{\varepsilon}+\varepsilon \Delta^2 \partial_{t}u^{\varepsilon}
=-u^{\varepsilon}_{xxx}+u^{\varepsilon}_{x}+\partial_{x}^{-1}u^{\varepsilon}_{yy}-\partial_{x}
\big[(u^a+u^{\varepsilon})^p-(u^a)^p\big]+\partial_x G,\quad u^{\varepsilon}(0)=0,
\end{equation}
where $\varepsilon>0$ and $\Delta=\partial_x^2+\partial_y^2$.
Thanks to the regularization (which is by no means canonical), we may solve
(\ref{four_epsilon}) for a finite time, independent of $\varepsilon>0$, by means of
the Picard iteration applied to the associated integral equation. 
We next wish to pass to the limit $\varepsilon\rightarrow 0^+$ in
$u^{\varepsilon}$. For that purpose, we need to establish an apriori bound on
$\|u^{\varepsilon}(t,\cdot)\|_{H^s}$ independent of $\varepsilon$. 
 These bounds follow from the fact that the well-chosen perturbation
  enjoy the same  $H^s$ estimate as the one formally obtained for
   \eqref{four}.
Then we pass to the limit
$\varepsilon\rightarrow 0^+$ thanks to a compactness argument.
This establishes the local well-posedness of (\ref{four}).

Finally, let us notice that the sufficient condition of Lemma~\ref{eigen} for the existence of an unstable
mode applies. Indeed, we have
$$M_{k}=  \partial_{x}L \partial_{x}  - k^2 Id, $$
 therefore,  it suffices to show that
the self adjoint operator $\partial_{x}L\partial_{x}$ on $L^2(\R)$, with domain $H^4(\R)$,
has a unique positive eigenvalue. Note that,  thanks to Weyl's theorem the essential
spectrum of $\partial_{x}L\partial_{x}$ is $]-\infty,0]$. Therefore on
$[0,\infty]$ the spectrum of $\partial_{x}L\partial_{x}$ can only contains eigenvalues
 of finite multiplicity and hence for $k>0$, $M_{k}$ is Fredholm with zero index.
Since $L$ has a unique negative eigenvalue and the remainder of its spectrum 
is included in $[0,\infty]$,               
we obtain that, by analyzing the corresponding quadratic forms, the operator
$\partial_{x}L\partial_{x}$ cannot have more than one positive direction,
i.e. $u$ such that $(u,\partial_{x}L\partial_{x}u)>0$.
Let us finally show that a positive direction indeed exists.
Let us denote by $u_{-1}$ the $L^2$ normalized eigenvector of $L$ with corresponding to the negative
eigenvalue $-\mu$. Let $\varphi_{n}\in H^{10}(\R)$ be a sequence such that
$\partial_x \varphi_n$ converges to $u_{-1}$ in $H^{10}(\R)$.
Then $(\partial_{x}L\partial_{x}\varphi_n,\varphi_n)$ converges to
$\mu>0$. Therefore there exist a   positive direction of
$\partial_{x}L\partial_{x}$ which shows that $\partial_{x}L\partial_{x}$ has a
positive eigenvalue (recall that on $\R^+$ the spectrum of $\partial_{x}L\partial_{x}$ can only contains eigenvalues).

 Finally,  for $ k_{0}^2$ the unique positive eigenvalue of $\partial_{x}L \partial_{x}$,
 we have
 $$ \Big([{d\over dk} M_{k} ]_{k=k_{0}} \varphi, \varphi \Big) = -2k_{0} \neq 0$$
  and hence \eqref{hyplem} is verifed. Thus Lemma~\ref{eigen} applies in the context of (\ref{gKPI}).

Therefore we can apply our general theory and obtain that 
$Q$ is (orbitally) unstable as a solution of (\ref{gKPI}) (posed on
$\R\times\T_{a}$ with a suitable $a$ or $\R^2$)  
thanks to our general results. We have the following statement.
\begin{theoreme}
For every $s\geq 0$, there exists $\eta>0$ such that for every $\delta>0$ there
exists $u_0^\delta$ and a time $T^\delta\sim |\log\delta|$ such
that 
$
\|u_0^\delta-Q\|_{H^s(\R^2)}<\delta
$
and the generalized KP-I equation (\ref{gKPI}) with data
$u_0^\delta$ is locally well-posed on $[0,T^\delta]$. Moreover, if we denote by
$u^\delta(t)$, $t\in [0,T^\delta]$, the corresponding solution, then
 $u^\delta(t) - Q \in H^s(\mathbb{R}^2)$ for every $t \in [0, T^\delta]$ and 
$$
\inf_{v\in {\mathcal F}}\|u^\delta(T^\delta)-v\|_{L^2(\R^2)}\geq \eta,
$$
where $\mathcal{F}$ is the space of $L^2(\R)$ functions independent of $y$.
\end{theoreme}
A similar statement may be done for periodic in $y$ solutions with a suitable period
depending on the transverse frequency of the unstable mode (see Theorem~\ref{theoper} above).

Let us recall (see \cite{Liu,Saut}) that for $p=3,4$ the generalized KP-I equation, posed on $\R^2$
has local smooth solutions blowing up in finite time, i.e. another (stronger)
type of instability exists in these cases. This is in sharp contrast with the
case $p=2$, i.e. the ``usual'' KP-I equation when global smooth solutions
exist both in the case of data periodic in $y$ (see \cite{IK}) or localised
with respect to $Q$ (or zero), see \cite{MST}.

\subsection{The nonlinear Schr\"odinger equation}
The $1d$ model is 
$$
(i\partial_{t}+\partial_{x}^{2})u+|u|^{2}u=0,\quad u:\R\rightarrow \C.
$$
This equation has a solitary wave  solution of the form $u(t,x)=e^{it}Q(x)$ with $Q$  smooth
with exponential decay. 
More precisely
$
Q(x) = \sqrt{2}(\mbox{ch}(x))^{-1}.
$
Then after  changing $u$ in $e^{it} u$,  $Q$ becomes a  stationary solution of
\begin{equation}\label{nls}
(i\partial_{t}+\partial_{x}^{2})u-u+|u|^{2}u=0,\quad u:\R\rightarrow \C.
\end{equation}
By writing $u=u_1+i u_2$ with real valued $u_1,u_2$, we obtain that $U\equiv
(u_1,u_2)^t$ solves the equation
\begin{equation}\label{nls_pak}
\partial_t U=
\left( 
\begin{array}{cc} 0 & 1 \\ 
-1   &  0
\end{array}
\right)
\Big(
(-\partial_x^2+1)U+\nabla F(U)
\Big),\quad F(U)=-\frac{1}{4}(u_1^2+u_2^2)^2
\end{equation}
which fits in our framework with 
$$
d=2,\quad 
J=\left( 
\begin{array}{cc} 0 & 1 \\ 
-1   &  0
\end{array}
\right),\quad
L_0=-\partial_x^2+{\rm Id}\,.
$$
The solution $Q$ of (\ref{nls}) is orbitally stable (see \cite{CL}).
This implies the orbital stability of $(Q,0)^t$ as a solution of (\ref{nls_pak}).
The operator $L$ in the context of (\ref{nls_pak}) is given by
$$
L=\left( 
\begin{array}{cc} -\partial_x^2+{\rm Id}-3Q^2 & 0 \\ 
0   &  -\partial_x^2+{\rm Id}-Q^2
\end{array}
\right).
$$
The spectral condition \eqref{spectral} on $L$ is satisfied since $-\partial_x^2+{\rm Id}-3Q^2$
has exactly two simple eigenvalues $-3$ and $0$ with corresponding
eigenvectors $Q^2$ and $Q'$ and continuous spectrum $[1,\infty[$ 
while $-\partial_x^2+{\rm Id}-Q^2$ has one simple eigenvalue $0$ with
corresponding eigenfunction $Q$ and  continuous spectrum $[1,\infty[$ 
(see e.g. \cite{T,W}). 

The transversely perturbed model is the 2D NLS equation that we can write
\begin{equation}\label{ven}
\partial_t U=
\left( 
\begin{array}{cc} 0 & 1 \\ 
-1   &  0
\end{array}
\right)
\Big(
(-\partial_x^2+1)U+\nabla F(U)-\partial_{y}^2U
\Big),\quad F(U)=-\frac{1}{4}(u_1^2+u_2^2)^2,
\end{equation}
i.e. $\mathcal{S}(\partial_{y})=-\partial_{y}^{2}$ and $S(ik)=k^2{\rm Id}$.
The assumptions of sections \ref{sJ}, \ref{sS} are easy to check. 
Since $Q$ is bounded,  assumption (\ref{klarge}) is also  trivially satisfied.

    Let us next turn to the assumption on the key eigenvalue problem in the
context of (\ref{ven}).
 Again, we shall use the criteria of section \ref{criteria}.
  This is very simple in this case, since
   $ \sigma - J(L+ S(ik))$ is already a differential operator.
    Consequently, we can take  $R(\sigma, k)= {\rm Id}.$
If we introduce $V=(u_{1}, u_{2}, \partial_{x}u_{1}, \partial_{x}u_{2})^t \in \mathbb{C}^4$,
$\mathbb{F}=(0, 0, F_2, F_1)^t$ we can  rewrite the resolvent equation as
$
V_{x} = A(x,\sigma,k) V + \mathbb{F},
$
where for all $k$,
$$
A(x,\sigma,k)=
\left( 
\begin{array}{cccc}
0 & 0 & 1 & 0 \\ 
0 & 0 & 0 & 1 \\
k^2+1-3Q^2 & \sigma & 0 & 0 \\
-\sigma & k^2+1-Q^2 & 0 & 0
\end{array} 
\right).
$$
Thus
$$
A_{\infty}(\sigma,k)=
\left( 
\begin{array}{cccc}
0 & 0 & 1 & 0 \\ 
0 & 0 & 0 & 1 \\
k^2+1 & \sigma & 0 & 0 \\
-\sigma & k^2+1 & 0 & 0
\end{array} 
\right)
$$
and we see that $A$ and $A_{\infty}$ are analytic in $(\sigma, k)$.
We then define $D(\sigma,k)$ as the Wronskian associated to
$A(x,\sigma,k)$. Thus $\tilde{D}(\sigma,0)=D(\sigma,0)$ in the case of the NLS.
The eigenvalues  of $A_{\infty}(\sigma,k)$  are the roots of the  polynomial $P$ 
\beq\label{Pbis}
P(\lambda ) = (\lambda^2-k^2-1)^2+\sigma^2.
\eeq
Therefore, in the context of (\ref{ven}), for {\it every} $k\in\R$ the spectrum of $A_{\infty}(\sigma,k)$ does
not meet the imaginary axis.  
Thus the assumption \eqref{compensation}
is obviously satisfied.
  Moreover,  since $\tilde{D}(\sigma, 0)= D(\sigma, 0)$, 
   \eqref{evans1D} and \eqref{evans1Dt} are met  because of
    the 1D stability of the solitary wave.

Since $J$ is a  zero order operator and $L_{0}$ has
 the required form, we can use 
 Corollary \ref{cormult} to get the existence of
   a multiplier.

The nonlinear assumptions in the context of (\ref{nls}) is satisfied thanks to
 the standard well-posedness argument for the 2D NLS equation 

Moreover, since here $J$ is of order zero, the estimate \eqref{tame} follows
 by the standard Gagliardo-Nirenberg-Moser inequalities. 

Finally, 
the sufficient condition given by Lemma \ref{eigen} for the existence of an unstable mode applies. Indeed, as for the KP equation, we have
$$M_{k}= JLJ - k^2 Id$$
 since
$$
JLJ=-\left( 
\begin{array}{cc} -\partial_x^2+{\rm Id}-Q^2 & 0 \\ 
0   &  -\partial_x^2+{\rm Id}-3Q^2
\end{array}
\right)
$$
which have a unique positive eigenvalue. The non-degeneracy condition
 \eqref{hyplem} is  also obviously verified.

Therefore, our general theory applies and we can state the following results.
\begin{theoreme}
For every $s\geq 0$, there exists $\eta>0$ such that for every $\delta>0$ there
exists $u_0^\delta$ and a time $T^\delta\sim |\log\delta|$ such
that 
$
\|u_0^\delta-Q\|_{H^s(\R^2)}<\delta
$
and the two dimensional NLS equation
$$
(i\partial_{t}+\partial_{x}^{2}+\partial_{y}^2)u+|u|^{2}u=0,\quad u:\R^2\rightarrow \C
$$
with data $u_0^\delta$ is locally well-posed on $[0,T^\delta]$.  If we denote by
$u^\delta(t)$, $t\in [0,T^\delta]$,  the corresponding solution,  then
 we have $u^\delta(t) - Q
 \in H^s(\mathbb{R}^2)$, $\forall t \in [0, T^\delta ]$  
and
$$
\inf_{v\in {\mathcal F}}\|u^\delta(T^\delta)-v\|_{L^2(\R^2)}\geq \eta,
$$
where $\mathcal{F}$ is the space of $L^2(\R)$ functions independent of $y$.
\end{theoreme}
A similar statement may be done for periodic in $y$ solutions with a suitable period
depending on the transverse frequency of the unstable mode (see Theorem~\ref{theoper} above).

\subsection{The  Boussinesq equation}
Consider the $1d$  Boussinesq equation 
\begin{equation}\label{bous1d}
u_{tt}+(u_{xx}+u^2-u)_{xx}=0,
\end{equation}
This  equation has a traveling wave solution (see \cite{BS}) of the form
$$
 u(t,x) = q(x-ct) \equiv q\in H^{\infty}(\R;\R^2),\quad |c|<1, \, c \neq 0.
$$
In addition $q$ has an exponential decay at infinity.
Note that we have
 $$q(x)= (1 - c^2 )  \,Q^{KdV}( \sqrt{  1- c^2  } \, x)$$
 where $Q^{KdV}$ is the solitary wave with unit speed of the KdV
  equation given by \eqref{QKP} (for $p=2$).
Moreover for $|c|\in ]1/2,1[$ this traveling wave is orbitally stable (see \cite{BS}).

At first, we shall  
 rewrite  (\ref{bous1d})  as a first order equation.
  Let us define 
  $$ Bu = -u_{xx}+ u$$
  and $B^\alpha$ as the Fourier multiplier with symbol
   $ (|\xi|^2 + 1)^\alpha$. Note that $B^\alpha$ is a symmetric operator.
    By using $B$, we rewrite
    \eqref{bous1d} as
    $$ u_{t}= \partial_{x} B^{\frac{1}{2} } v, \quad
     v_{t}=  \partial_{x} B^{ - \frac{1}{2} }\Big(  Bu  - u^2 \Big).$$
  Changing $x$ into $x-ct$, we get
  \beq
  \label{bous1dS} u_{t}=\partial_{x}B^{- { 1 \over 2 } } \Big( Bv + c B^{ 1 \over 2 } u \Big)
  , \quad v_{t}=  \partial_{x} B^{ - { 1 \over2} }\Big(  Bu  + c B^{ 1 \over 2} v - u^2 \Big).
  \eeq
  With this change of frame, $Q(x)= (q(x), -cB^{ - { 1 \over 2 } } q(x))$
   is a stationary solution of \eqref{bous1dS}. 
    By setting $U= (u, v)^t$, 
 $$ J =  \left( \begin{array}{cc} 0 & \partial_{x} B^{- { 1 \over 2} } \\
  \partial_{x} B^{ -{ 1 \over 2 } } & 0  \end{array} \right), \quad
    L_{0}= \left( \begin{array}{cc}  B & c B^{ 1 \over 2 } \\
    c B^{ 1 \over 2 } & B \end{array} \right), \quad
     F(U) =\left( \begin{array}{ll} -u^{3}/ 3  \\ 0 \end{array} \right)
     $$  
we can write \eqref{bous1dS} under the form \eqref{ham1pak}.
 We easily check that the assumptions of
  section \ref{s1D} are matched. Note that by Bessel identity, we have
  $$ (L_{0} U, U) =(2\pi)^{-1} \int_{\R} \Big( ( 1+ |\xi|^2) |\hat{u}(\xi)|^2 + 
    ( 1+ |\xi|^2) |\hat{v}(\xi)|^2 + 2c \, (\sqrt{ 1 + | \xi |^2 }\, \hat{u}(\xi)\, \overline{\hat{v}(\xi)}\Big)\,
     d \xi$$
      and hence, \eqref{coerc} is verified for $c<1$. Moreover, an important remark
       is that  $J$ is here a bounded operator, in contrast with the
       formulation used by \cite{BS}.
       
  Next, let us check \eqref{spectre}. The operator $L$ is defined by $L=L_0+R$, where 
  $$
  R = 
 \left( \begin{array}{cc} -2q  & 0 \\ 0  & 0 \end{array} \right).$$
The spectral condition on $L$ is again satisfied  thanks to the Sturm-Liouville theory
 and  is proven for  $ 1/2 < |c|<1$  in \cite{BS} in order to prove the nonlinear stability of
  the  solitary wave. Note that the formulation \eqref{bous1dS} that we
  use is equivalent to the one used   by Bona and Sachs in \cite{BS}.
 

The transversally perturbed model is
\beq
\label{bous2d}
u_{tt}+(u_{xx}+u^2-u -   \partial_{x}^{-2} \partial_{y}^2 u )_{xx}=0.
\eeq
The equation (\ref{bous2d}) has been derived in \cite{Johnson} as a model for
interacting shallow water waves.
Again, to rewrite this equation as a first order system, we introduce
$$ \mathcal{B}w= -w_{xx} + w  + \partial_{x}^{-2} w_{yy}.$$
 We write \eqref{bous2d} as
\beq
\label{Bous2d}
 u_{t}= \partial_{x} \mathcal{B}^{ 1 \over 2 }v, \quad  v_{t}= \partial_{x}
 \mathcal{B}^{-\frac{1}{2}} \Big( \mathcal{B}u -  u^2 \Big)
 \eeq
 and hence going into the moving frame, we find
 $$ 
u_{t}= \partial_{x} \mathcal{B}^{  - {1 \over 2} } \Big(\mathcal{B}v +c \mathcal{B}^{1\over2}
 u \Big)
 , \quad  v_{t}= \partial_{x}
 \mathcal{B}^{-{ 1 \over 2 } } \Big( \mathcal{B}u  + c \mathcal{B}^{ 1 \over 2 } v-  u^2 \Big).
 $$
 Consequently, we  get a system under the form \eqref{ham2} with
$$
\mathcal{J}(\partial_{y})=\left( \begin{array}{cc} 0 & \partial_{x}\mathcal{B}^{- {1 \over 2 } } \\
   \partial_{x} \mathcal{B}^{- { 1 \over 2 } } & 0 \end{array} \right)
,\quad
\mathcal{S}(\partial_{y})=
\left( \begin{array}{cc}  \partial_{x}^{-2}\partial_{y}^2 & 
     c \big( \mathcal{B}^{ 1 \over 2} - B^{ 1 \over 2}\big) \\
      c \big( \mathcal{B}^{ 1 \over 2} - B^{ 1 \over 2}\big) & \partial_{x}^{-2}\partial_{y}^{2}
       \end{array}\right).
$$
Therefore the $1d$ operators $J(ik)$ and $S(ik)$ are defined as 
$$
  J(ik) = \left( \begin{array}{cc} 0 & \partial_{x}B(ik)^{- {1 \over 2 } } \\
   \partial_{x} B(ik)^{- { 1 \over 2 } } & 0 \end{array} \right)
   , \quad
     S(ik) = \left( \begin{array}{cc}  - k^2 \partial_{x}^{-2} & 
     c \big( B(ik)^{ 1 \over 2} - B^{ 1 \over 2}\big) \\
      c \big( B(ik)^{ 1 \over 2} - B^{ 1 \over 2}\big) & -k^2 \partial_{x}^{-2}
       \end{array}\right)$$
    with
    $$ B(ik)w = - w_{xx} + w - k^2 \partial_{x}^{-2}.$$
   Note that $J(ik)$ is a bounded operator on $L^2$. Indeed, its symbol is given by
    $$ \left( \begin{array}{cc} 0 & { i \xi \over \big(  1 + \xi^2 + { k^2 \over \xi^2 } \big)^{1
    \over 2}  } \\
 {  i\xi \over  \big(  1 + \xi^2 + { k^2 \over \xi^2 } \big)^{1
    \over 2} }   & 0 \end{array} \right).$$
    
 We can  easily check that the assumptions of section \ref{sJ} are verified. 
  Since  $J$  is bounded on $L^2$, the estimate
   \eqref{JR} and \eqref{Jik} are   obvious.  
      
We  also  easily check the assumptions of section \ref{sS}. The first
 three  assumptions can be  verified by computation. 
 Moreover, we notice that there exist positive constants $c_{0}$, $C_{0}$ such that
 \beq
 \label{sikbou} c_{0} k^2 |\partial_{x}^{- 1} U |^2  \leq \big( S(ik) U, U\big) \leq
  C_{0} k^2 |\partial_{x}^{- 1} U |^2.
  \eeq
  Indeed  by using the Fourier transform, we have
$$ \big( S(ik) U, U\big)
 =(2\pi)^{-1} 
\int_{\R}\Big(  {k^2 \over \xi^2}\big( |\hat{u}(\xi)|^2 + |\hat{v}(\xi)|^2 \big)
  +  2c  { {k^2 \over \xi^2} \over  \big( 1 + \xi^2 + {k^2 \over \xi^2} \big)^{ 1 \over 2} 
   + \big( 1 + \xi^2 \big)^{ 1 \over 2}  } \hat{u}(\xi)\, \overline{\hat{v}(\xi)} \Big) d\xi$$
   and hence \eqref{sikbou} follows since $c<1$. Thanks to \eqref{sikbou}, we also  find
     that \eqref{JSw} is verified with $C(k)= k $.
     
   Next, to get \eqref{klarge},  we use again the Fourier transform
    to write
  \begin{equation*}
 (LU, U) + (S(ik) U, U) \geq C_1 \int_{\R}(1-c ) \big(  1 + \xi^2 + { k^2 \over \xi^2 } - C \big)
    |\hat{U}(\xi)|^2  d\xi,
    \end{equation*}
    where $C_1\approx |Q|_{L^{ \infty}}$. Consequently, we get
     \eqref{klarge} as for the KP equation.
 
 The assumptions of section~\ref{Ms} on the existence of  suitable multipliers
  follows again from Corollary \ref{cormult}.  Indeed, 
   $J(ik)$ is a zero order operator and $L_{0}= - \partial_{x}^2
    + \tilde{L}$ with $\tilde{L}$ a first order operator.    
  
Let us turn to the study of the resolvent equation \eqref{res2}.
 We first  notice that
 $$ \sigma{\rm Id} - J(ik) ( L+ S(ik) ) = 
  \left( \begin{array}{cc} \sigma - c \partial_{x} &  - \partial_{x}B(ik)^{ 1 \over 2 } \\
   - \partial_{x}\, B(ik)^{ - { 1 \over 2 } } \big( B(ik)  - 2 q \big) &
    \sigma - c \partial_{x} \end{array} \right).$$
Consequently,  by using again section \ref{criteria}, we can set
$$ R(\sigma, k) = \left( \begin{array}{cc}  \sigma - c \partial_{x } &  \partial_{x}
 B(ik)^{ 1 \over 2} \\ 0 & 1 \end{array} \right)$$
  to get  \eqref{factorisation} with
 \begin{eqnarray*}
 & &P_{1}
    (\sigma, k)u  = 
      \partial_{x}^4 u  - \partial_{x}^2 u  + k^2 u  + 2 \partial_{x}^2 ( qu) +
     (\sigma - c \partial_{x})^2 u, \\
 & & E(\sigma, k) =  \sigma -c \partial_{x}, \\
 & &  P_{2} (\sigma, k) u = - \partial_{x} B(ik)^{ - { 1 \over 2 }} 
 \big( B(ik) u  - 2 q u \big) = - \partial_{x} B(ik)^{1 \over 2 } u - 2 
 \partial_{x} B(ik)^{ - { 1 \over 2 }}(qu).
 \end{eqnarray*}
 Consequently, $P_{1}$ is a fourth order differential operator analytic in $(k, \sigma)$
  for 
  every $k$, $E$ is invertible for $\mbox{Re }\sigma > 0$ and
   $P_{2}$ is a second order operator with domain $H^2$.
     Indeed, $\partial_{x} B(ik)^{ - { 1 \over 2 }}$
      is a bounded operator on $L^2$ and we have the estimate
$$
2\pi \|  \partial_{x} B(ik)^{1 \over 2 } u \|^2
 = \int_{\R} \xi^2  \big( \xi^2 + 1 + { k^2 \over \xi^2} \big)  |\hat{u}(\xi) |^2
 \, d\xi \leq  \int_{\R} \big(  1+ \xi^4 + k^2  \big)   |\hat{u}(\xi) |^2
d\xi,
$$
 which is uniform  for $k$ in a vicinity of zero.   

 One can rewrite \eqref{P2eq} in the context of the Boussinesq equation as a first order system \eqref{edo} with
  $A(x, \sigma, k) \in M_{4}(\mathbb{C})$ for every $k$. The assumption
  \eqref{asympt}  is met since $q$ decays exponentially fast to zero at infinity.
  Moreover,  the eigenvalues $\lambda$  of $A_{\infty}(\sigma, k)$
 are the   roots of the
polynomial $P$ defined as
$$
P(\lambda)=\lambda^4-(1-c^2)\lambda^2-2c\sigma\lambda+k^2+\sigma^2.
$$
Suppose that $P$ has a root of the form $\lambda=i\mu$ with $\mu\in\R$. Then
by separating the real and the imaginary part of $P(i\mu)$, we get the
relations
$$
\mu^4+(1-c^2)\mu^2+2c\sigma_2\mu+\sigma_1^2-\sigma_2^2+k^2=0,\quad
-2c\sigma_1\mu+2\sigma_1\sigma_2=0,
$$
where $\sigma=\sigma_1+i\sigma_2$, $\sigma_1,\sigma_2\in\R$.
Therefore, since $\Re(\sigma)=\sigma_1\neq 0$, we have that $\mu=\sigma_2/c$
(recall that we are interested for the values of $c$ such that $1/2<|c|<1$). 
By substituting the value of $\mu$ in the first equation, we get
\begin{equation}\label{cergy}
\frac{\sigma_2^4}{c^4}+\frac{\sigma_2^2}{c^2}+\sigma_1^2+k^2=0.
\end{equation}
But since  in the last equation
 for $\sigma_{2}$, if $\sigma_{2}$ is real, all the terms are non-negative and $\sigma_{1}^2>0$,
  there is  no real
root for every $k \in \mathbb{R}$.
Therefore for every $k\in\R$ and $\Re(\sigma)\neq 0$ the equation
$P(\lambda)=0$ has no root on the imaginary axis. Since for $k=0$ there is no
complication coming from the emergence of a root on the imaginary axis,
we are in the same situation as for  the nonlinear Schr\"odinger equation.
The assumption \eqref{evans1D} (and hence also \eqref{evans1Dt}) is met  since  the solitary wave
  $q$ is stable as a solution of the $1D$ Boussinesq equation for
   $1/2< |c| < 1$  as shown
   in \cite{BS}, we get \eqref{resper}, \eqref{resloc} from Lemma~\ref{lemper}
   and Lemma~\ref{lemloc}.


The ``nonlinear'' assumptions  of section  \ref{asnon}
 are also met.
 Indeed, the local well-posedness  of the 
  $2d$ Boussinesq equation which is semi-linear can be obtained
   by standard techniques. Moreover,  since
    $\mathcal{J}$ is a bounded operator on $H^s$, the assumption \eqref{tame}
    follows readily from  the Gagliardo-Nirenberg-Moser inequality.

\begin{remarque}
{\rm
Let us observe that if we consider transverse perturbation with the opposite
sign that even the problem defining the free evolution is ill-posed in Sobolev
spaces. Thus in the context of the Boussinesq equation the analogue of the
KP-II equation is not a ``good'' model in Sobolev spaces.
Indeed consider the linear problem
\begin{equation}\label{ill}
u_{tt} + \big( u_{xx}  - u + \partial_{x}^{-2} \partial_{y}^2 u \big)_{xx}= 0,
\end{equation}
 Using the
Fourier transform, we obtain that $\hat{u}$ solves  
$$
 \hat{u}_{tt} + \big( \xi^4 + \xi^2 - k^2 \big) \hat{u} = 0,$$
 and hence, one can find growing modes  $ e^{\lambda t } \hat{u}(\xi, k)$
    if 
  $$
P(\lambda)=\lambda^2 + \xi^4 + \xi^2 - k^2 = 0\,.
$$
 Since  one can find roots of $P$    with arbitrary large real parts
and arbitrary sign  this 
implies the ill-posedness of (\ref{ill}).
A similar phenomenon occurs for the equation
\begin{equation}\label{ill-bis}
u_{tt}+(-u_{xx}-u-\partial_{x}^{-2}\partial_{y}^{2})_{xx}=0,
\end{equation}
where the sign is changed in front of the dispersion term.
In view of this discussion, it becomes reasonable to study (\ref{ill}) or
(\ref{ill-bis}) in analytic spaces.}
\end{remarque}

 We can also use Lemma \ref{eigen} to get an unstable eigenmode.
  Indeed, we have
  $$ M_{k}= \left( \begin{array}{cc} \partial_{xx} & c \partial_{xx} B(ik)^{- { 1 \over 2 } } 
  \\  c \partial_{xx} B(ik)^{- { 1 \over 2 } } &   \partial_{x}B(ik)^{- { 1 \over 2 } } \big(
   B(ik) - 2q \big) \partial_{x} B(ik)^{- { 1 \over 2 } } \end{array}\right).$$
   Consequently $(u,v)^t\in L^2(\R;\R^2)$ is in the kernel of $M_{k}$ if and only if
 $$  u  = -c  B(ik)^{- { 1 \over 2 } } v, \quad
  \partial_{x } B(ik)^{- { 1 \over 2 } } \Big(
   - c^2 +  B(ik) - 2 q \Big)\partial_{x}   B(ik)^{- { 1 \over 2 } } v = 0. 
   $$
   Next, we notice that
   \begin{multline*}
     \partial_{x} B(ik)^{- { 1 \over 2 } } \Big(
   - c^2 +  B(ik) - 2 q \Big)\partial_{x}   B(ik)^{- { 1 \over 2 } } = 
\\ 
=
B(ik)^{- { 1 \over 2 } } \Big(  \partial_{x} \big(- \partial_{x}^2 + (1 - c^2) - 2 q \big)
    \partial_{x } - k^2 \Big) B(ik)^{- { 1 \over 2 } } 
\equiv   B(ik)^{- { 1 \over 2 } } m_{k}\,  B(ik)^{  - {1 \over 2} }. 
    \end{multline*}
Notice that $m_{k}$ is the operator $JLJ+ JSJ$ which appears
     in the study of the stability of the solitary wave with  speed $1-c^2$
     of the KP-I equation. Thus as in the analysis for the KP-I equation, we can show that
$m_{k}$ has
      a one-dimensional  non trivial kernel for some $k_{0} \neq 0$. This implies
       that $M_{k_{0}}$ also has a non-trivial kernel  generated by  
       $$ \varphi= (  - c \psi, B(ik_{0})^{ 1 \over 2 } \psi),$$
       $\psi$ being nontrivial and such that $m_{k_{0}} \psi = 0 $.
        Moreover, we can deduce that $M_{k_{0}}$ is Fredholm
         index 0 from the fact that $m_{k_{0}}$ is Fredholm index 0.
        Let us check the non-degeneracy condition \eqref{hyplem}.
      Using the identity
    \begin{equation}\label{arx-1}
     { d\over dk}_{/k= k_{0}} B(ik)^{-\frac{1}{2}}  =   k_{0} \partial_{x}^{-2} B(ik_{0})
      ^{ - { 3 \over 2 }},
\end{equation}
      we obtain that
   \begin{equation}\label{arx}
\big(  \Big[{d\over dk} M_{k}\Big]_{k=k_0} \varphi, \varphi \big)
   = -2 k_{0} |\psi |^2\neq 0
\end{equation}  
since $k_{0}\neq 0 $. More precisely $M_{k}=M_{k}^{1}+M_{k}^{2}$ with
$$ 
M_{k}^{1}= 
\left( \begin{array}{cc} \partial_{xx} & 0
  \\  0 &   \partial_{x}B(ik)^{- { 1 \over 2 } } \big(
   B(ik) - 2q \big) \partial_{x} B(ik)^{- { 1 \over 2 } } \end{array}\right)
$$
and
$$
M_{k}^{2}=\left( \begin{array}{cc} 0 & c \partial_{xx} B(ik)^{- { 1 \over 2 } } 
  \\  c \partial_{xx} B(ik)^{- { 1 \over 2 } } &   0
\end{array}\right).
$$
A use of  (\ref{arx-1}) gives
$$
\big(  \Big[{d\over dk} M_{k}^2\Big]_{k=k_0} \varphi, \varphi \big)=
 - 2 c^2  k_{0}(B(ik_{0})^{-1} \psi, \psi)
$$
and (using that  $m_{k_{0}}\psi=0$)
$$
\big(  \Big[{d\over dk} M_{k}^1\Big]_{k=k_0} \varphi, \varphi \big)=
2 c^2  k_{0}(B(ik_{0})^{-1} \psi, \psi)-2 k_{0} |\psi |^2\,.
$$
Thus the identity (\ref{arx}) indeed holds true. This allows to use Lemma \ref{eigen} to get the existence of an unstable
  eigenmode.

Consequently, we can apply our general theory to get the following statement.
\begin{theoreme}
Consider the equation \eqref{Bous2d} for $|c| \in (1/2, 2)$. 
For every $s\geq 0$, there exists $\eta>0$ such that for every $\delta>0$ there
exists $u_0^\delta$ and a time $T^\delta\sim |\log\delta|$ such
that 
$
\|u_0^\delta-Q\|_{H^s(\R^2)}<\delta
$
 and the solution $u^\delta (t)$  of \eqref{Bous2d}
with data $u_0^\delta$ is   defined on $[0,T^\delta]$, with   $u^\delta(t) - Q
 \in H^s(\mathbb{R}^2)$, $\forall t \in [0, T^\delta ]$  
and moreover satisfies the estimate
$$
\inf_{v\in {\mathcal F}}\|u^\delta(T^\delta)-v\|_{L^2(\R^2)}\geq \eta,
$$
where $\mathcal{F}$ is the space of $L^2(\R)$ functions independent of $y$.
\end{theoreme}
A similar statement may be done for periodic in $y$ solutions with a suitable period
depending on the transverse frequency of the unstable mode (see Theorem~\ref{theoper} above).

\subsection{The Zakharov-Kuznetsov equation}
The Zakharov-Kuznetsov equation
\begin{equation}\label{ZK}
u_{t}+u_{xxx}+u_{xyy}+uu_x=0
\end{equation}
is derived in \cite{ZK} to describe the propagation of nonlinear ionic-sonic
waves in plasma magnetic field. 
Equation (\ref{ZK}) is a two dimensional generalization of the KdV equation
which fits into our general framework with $d=1$. 
Indeed, if we denote by $Q$ the suitable speed one KdV solitary wave, then $Q$ is a
stationary solution of 
\begin{equation}\label{ZKbis}
u_{t}-u_{x}+u_{xxx}+u_{xyy}+uu_x=0.
\end{equation}
We can write (\ref{ZKbis}) as
$$
u_{t}=J(L_0+\nabla F(u)+{\mathcal S}(\partial_y))u
$$
with
$$
J=\partial_{x},\quad L_{0}=-\partial_{x}^{2}+{\rm Id},\quad
F(u)=-\frac{u^3}{6},\quad {\mathcal S}(\partial_y)=-\partial_{y}^2\,.
$$
Assumptions  of section \ref{s1D} are still verified since as for the KP-I
 equation,  the 1d model
 is  the KdV equation. Assumptions \ref{sJ}, \ref{sS}, \eqref{klarge} are easy to check.

 The operator
 $$ \sigma u  - J(L + S(ik))u  = \sigma u - u_{x} + u_{xxx} - k^2 u_{x} +  2 (Qu)_{x}$$
  is already a differential operator and hence we can readily  use
   section \ref{criteria} with $R(\sigma, k)= {\rm Id}.$ In
    particular, we have for every $k$
 $$ A(x, \sigma, k) =
  \left( \begin{array}{ccc} 0 & 1 & 0 \\ 0 & 0 & 1 \\ - \sigma - 2 Q_{x} & (1+ k^2)
   - 2Q & 0 \end{array}\right).
   $$
  This allows to find that the eigenvalues of $A_{\infty}(\sigma, k)$
   are the roots of the polynomial 
  $$P(\lambda)=  \lambda^3-(1+k^2)\lambda+\sigma.$$
   Since, for $\xi \in \mathbb{R}$, we have  $\mbox{Re }P(i\xi)= \mbox{Re }\sigma$,
    we get that for $\mbox{Re }\sigma\neq 0$ there is no eigenvalue of
    $A_{\infty}(\sigma, k)$ on the imaginary axis. 
    As  in the case
  of  NLS, the assumptions of section \ref{reshyploc} are obviously verified
   since we are in a situation where $A(x, \sigma, k)$
    is analytic for every $k$ and  where  there is no eigenvalue
     of $A_{\infty}$ on the imaginary axis even for $k=0$. Moreover, 
      \eqref{evans1D} (and hence also \eqref{evans1Dt})
       are verified thanks to the stability of the KdV solitary wave.
    Consequently,  the assumptions of sections \ref{hypper} and \ref{hyploc}
     follow from Lemma~\ref{lemper} and  Lemma~\ref{lemloc}.
     
 To check assumption \ref{Ms}, we use again Lemma \ref{lemmult}.
  By taking
  $$ K_{s} u =  -  { 2 \over 3 } ( 1 + 2 s) Q\, u,$$
 we find that $E_{s}$ is a first order operator and hence the assumption of
 existence of multiplier of section~\ref{Ms} follows from Lemma~\ref{lemmult}.
  
  The assumptions of section \ref{asnon}  i.e  about the local  well posedness
   of the nonlinear equation  are again verified
  by standard arguments. Note that assumption \eqref{tame}
   was already checked in the study of the KP equation.
   
    Finally, we   note that it does not seem possible to use the
     simple criterion of  Lemma \ref{eigen} to prove the
     existence of an unstable eigenmode.
      Indeed, we have
     $$ M_{k}= \partial_{x}( L-k^2 ) \partial_{x}, \quad L= -u_{xx} + u - 2 Q. $$
     It is easy to prove that  there exists $k_{0}$
      such that $M_{k_{0}}$ has a non-trivial kernel.
       Nevertheless, here  $M_{k_{0}}$ is not a  Fredholm operator with index zero. 
 Fortunately, the existence of unstable modes was obtained in
  \cite{Bridges} by using more sophisticated arguments
   (i.e. the multisymplectic formulation of the equation). Consequently, we have the following result.
   
 \begin{theoreme}
Consider the equation \eqref{ZKbis}. 
For every $s\geq 0$, there exists $\eta>0$ such that for every $\delta>0$ there
exists $u_0^\delta$ and a time $T^\delta\sim |\log\delta|$ such
that 
$
\|u_0^\delta-Q\|_{H^s(\R^2)}<\delta
$
 and the solution $u^\delta (t)$  of \eqref{ZKbis}
with data $u_0^\delta$ is   defined on $[0,T^\delta]$  with   $u^\delta(t) - Q
 \in H^s(\mathbb{R}^2)$, $\forall t \in [0, T^\delta ]$  
and moreover satisfies the estimate
$$
\inf_{v\in {\mathcal F}}\|u^\delta(T^\delta)-v\|_{L^2(\R^2)}\geq \eta,
$$
where $\mathcal{F}$ is the space of $L^2(\R)$ functions independent of $y$.
\end{theoreme}
A similar statement may be done for periodic in $y$ solutions with a suitable period
depending on the transverse frequency of the unstable mode (see Theorem~\ref{theoper} above).

\subsection{KP-BBM}
Consider the generalized BBM equation
\begin{equation}\label{BBM}
u_{t}-u_{txx}+u_{x}+\partial_{x}(u^p)=0
\end{equation}
and the $2d$ generalization of KP type
\begin{equation}\label{KP-BBM}
u_{t}-u_{txx}+u_{x}+\partial_{x}(u^p)-\partial_{x}^{-1}u_{yy}=0.
\end{equation}
For $c>1$ there is a solitary wave solution of (\ref{BBM}) of the form
$u(t,x)=Q(x-ct)$.
 Again, we note that 
 $$ Q(x) =  (c-1)^{ 1 \over p-1 } Q^{KdV}\big( \sqrt{ 1 - { 1 \over c } } x \big).$$
 Then $Q(x)$ is a stationary solution of the equation
\begin{equation}\label{BBM2}
u_{t}-(c-1)u_{x}-u_{txx}+cu_{xxx}+\partial_{x}(u^p) - \partial_{x}^{-1}u_{yy}=0.
\end{equation}

Equation (\ref{BBM2}) may be written under the form 
$$
\partial_{t}u=J( L_{0} +  \nabla F(u) + \mathcal{S}(\partial_{y})) u ,
$$
where
$$
J=(1-\partial_x^2)^{-1}\partial_{x},\quad
L_{0}=  - c\partial_{xx}+(c-1)Id, \quad F(u)= -  \frac{1}{p+1}\int u^{p+1}, 
\quad S(ik) u =  -  k^2 \partial_{x}^{-2} \,.
$$
The corresponding operator $L$ is 
$$
L(u)=-cu_{xx}+(c-1)u- pQ^{p-1}u\,.
$$
Again, it is very easy to check the assumptions of sections \ref{s1D}, 
 \ref{sJ}, \ref{sS}, \ref{klarge}. To ensure that \eqref{spectral} and
  \eqref{evans1D} are verified, we restrict ourself to $p \leq 4$, in this
  case all the waves for $c>1$ are stable in the 1D model which is
   the BBM equation \cite{PW}, \cite{BSS}.
   
 Note that we are in a semilinear situation since  $J$ is a zero
  order operator (and even better). Consequently, the assumption of existence
  of multiplier of section~\ref{Ms} is verified thanks to Corollary
  \ref{cormult} (recall that for scalar problems it is straightforward). 
   The assumption~\ref{tame}  is met thanks to the Gagliardo-Nirenberg-Moser
     inequality. Again the local well-posedness assumed in section \ref{asnon}
       can be proven by standard methods. 
   
  To check \eqref{resper}, \eqref{resloc}, we can use section \ref{criteria}.
   Since
  $$ \sigma u- J \big( L + S(ik)\big)u=
   \sigma u -  (1-\partial_{x}^2)^{-1} \partial_{x} \big( - c \partial_{xx} u 
    + (c-1) u - p Q^{p-1} u  - k^2 \partial_{x}^{-2} u \big), $$
     we set 
  $$ R(\sigma, k) = \left\{ \begin{array}{ll} ( {\rm Id} - \partial_{x}^2) \partial_{x}, \quad \mbox{ if } k\neq 0, \\
   {\rm Id} - \partial_{x}^2 ,  \quad \quad \quad \!  \mbox{ if } k= 0. \end{array}\right.$$
   Then we directly  find that
    $ R(\sigma, k) - J( L+ S(ik)))=P_{1}(\sigma, k)$ is a differential operator
     of order $4$ for $k \neq 0$ and $3$ for $k=0$.
      Consequently, the assumption of section~\ref{factorisation} is matched
       with an empty second block.
       
 For $k \neq 0$, we have \eqref{edo} with 
$$
A(x,\sigma,k)= c^{-1}
\left( 
\begin{array}{cccc}
0 & c & 0 & 0 \\ 
0 & 0 & c & 0 \\
0 & 0 & 0 & c \\
-k^2 - p\partial_x^{2}(Q^{p-1}) & -\sigma - 2p\partial_x(Q^{p-1}) & c-1- pQ^{p-1} &  \sigma
\end{array} 
\right).
$$
Thus
$$
A_{\infty}(\sigma,k)=c^{-1}
\left( 
\begin{array}{cccc}
0 & c & 0 & 0 \\ 
0 & 0 & c & 0 \\
0 & 0 & 0 & c \\
-k^2 & -\sigma &c- 1 &  \sigma
\end{array} 
\right).
$$
The eigenvalues  of $A_{\infty}(\sigma,k)$  are the roots of the  polynomial $P$ 
\beq\label{PB}
P(\lambda ) =c \lambda^4  - \sigma \lambda^3 - (c-1) \lambda^2 + \sigma \lambda + k^2
\eeq
and hence are not purely imaginary when $\Re\, \sigma >0$.
Moreover, there are two of positive real part and two of negative real part.
For $k=0$, we have 
$$
A(x,\sigma, 0)= c^{-1}
\left( 
\begin{array}{cccc}
0 & c & 0  \\ 
0 & 0 & c   \\
-\sigma-p\partial_x(Q^{p-1}) & 1-pQ^{p-1} & 0
\end{array} 
\right)
$$
and thus
$$
A_{\infty}(\sigma, 0)=c^{-1}
\left( 
\begin{array}{cccc}
0 & c & 0 \\ 
0 & 0 & c  \\
-\sigma & c-1  &  \sigma 
\end{array} 
\right).
$$
The characteristic polynomial of $A_{\infty}(\sigma, 0)$ is
$p(\lambda)=c \lambda^3  - \sigma \lambda^2 - ( c- 1 ) \lambda + \sigma$ and thus for $\Re(\sigma)>0$ the eigenvalues of $A_{\infty}(\sigma, 0)$ do not meet the
imaginary axis.  This allows to use Lemma \ref{existevans} to get
 the existence of the Evans function.
 Finally, since the BBM solitary wave
is stable (see e.g. \cite{PW}, \cite{BSS}) for $p\leq 4$, $ c>1$, 
we have $D(\sigma, 0) \neq 0$ when $\Re\, \sigma >0$ and hence the assumption 
 \eqref{evans1D} is met. Consequently, \eqref{resper} follows from \eqref{lemper}

To handle the localized case, we note that  
when $k$, tends to zero,  there is a single root $\lambda=0$ of (\ref{P}) on the imaginary axis
 and hence, there is spectrum of $A_{\infty}(\sigma, 0^+)$ on the imaginary
  axis.
More precisely,  for $k\sim 0$ this root behaves  as
\beq
\label{muBBM}
 \mu(\sigma,k) \sim -\frac{k^2}{ \sigma}.
\eeq
Consequently,  there  is only  one of the negative real part roots of (\ref{PB}) which goes to
zero. Since $\mu(\sigma,k)$ is analytic, we can use the Gap lemma \cite{GZ}, \cite{KS}
to get the continuation of the Evans function. Moreover, 
 by using the same method as in the study of the gKP equation,
  we can also write the Evans function as
 $|\tilde{D}(\sigma,0)|= |c\sigma D(\sigma, 0)|$ where
$D(\sigma, 0)$ is the Evans function 
associated to the linearized BBM equation about the solitary wave.  Again, 
since the BBM solitary wave
is stable   we also have  that $\tilde{D}(\sigma, 0)$ does
 not vanish for $\mbox{Re } \sigma >0$ and hence, \eqref{evans1Dt}
  is verified. 
  Note that, \eqref{compensation} is also met in view of \eqref{muBBM} since
  $$ R(\sigma, k) J(ik) S(ik)=  \partial_{xx} (- k^2 \partial_{x}^{-2}) = -k^2.$$
 Therefore, \eqref{resloc}
  follows from Lemma \ref{lemloc}.
  
Finally, as for the gKP equation, the existence of an unstable eigenmode
  follows from Lemma~\ref{eigen}.  Indeed, we can write $M_{k}$
   under the form
$$
M_{k}= (1 - \partial_{x}^2)^{-1} m_{k} ( 1 - \partial_{x}^2)^{-1},
$$
  where
  $$m_{k}u =  c  \partial_{x}\Big( \partial_{x}( -  \partial_{xx} +
  {(c-1)\over c}{\rm Id} + {p\over c}Q^{p-1} {\rm Id} \Big)\partial_{x} - k^2.$$
Again,  the existence of a nontrivial kernel for $m_{k}$
    comes from the study of the KP equation and one can deduce that
     $M_{k}$ is Fredholm from the fact that $m_{k}$ is Fredholm.

  Therefore, we can state the following result.
 
  \begin{theoreme}
Consider the equation \eqref{BBM2} for $c>1$ and $p \leq 4$. 
For every $s\geq 0$, there exists $\eta>0$ such that for every $\delta>0$ there
exists $u_0^\delta$ and a time $T^\delta\sim |\log\delta|$ such
that 
$
\|u_0^\delta-Q\|_{H^s(\R^2)}<\delta
$
 and the solution $u^\delta (t)$  of \eqref{BBM2}
with data $u_0^\delta$ is   defined on $[0,T^\delta]$  with   $u^\delta(t) - Q
 \in H^s(\mathbb{R}^2)$, $\forall t \in [0, T^\delta ]$  
and moreover satisfies the estimate
$$
\inf_{v\in {\mathcal F}}\|u^\delta(T^\delta)-v\|_{L^2(\R^2)}\geq \eta,
$$
where $\mathcal{F}$ is the space of $L^2(\R)$ functions independent of $y$.
\end{theoreme}
A similar statement may be done for periodic in $y$ solutions with a suitable period
depending on the transverse frequency of the unstable mode (see Theorem~\ref{theoper} above).

Let us point out that the KP-BBM model considered in this section is not the
relevant one  from modelling view point (see \cite{Mameri}), the relevant
one being
\begin{equation}\label{BBM2pak}
u_{t}-(c-1)u_{x}-u_{txx}+cu_{xxx}+\partial_{x}(u^p) + \partial_{x}^{-1}u_{yy}=0.
\end{equation}
Equation (\ref{BBM2pak}) does not fit in the framework considered in this
paper and it is possible that the KdV soliton is in fact stable as a solution
of (\ref{BBM2pak}). Nevertheless  our KP-BBM model seems interesting
 for the following reason.
\subsection{Final remark}        
Let observe that in the case $p=2$ the equation (\ref{BBM2}) is globally
well-posed for data close to $Q$. We have therefore nonlinear instability in
the context of global well-posedness. Therefore this type of phenomena already
encountered in the context of the KP-I equation is not only
restricted to integrable models as the KP-I equation.
Let us briefly explain how we prove the global well-posedness for 
$$
u_{t}-(c-1)u_{x}-u_{txx}+cu_{xxx}+\partial_{x}(u^2) - \partial_{x}^{-1}u_{yy}=0
$$
with initial data
$$
u(0,x,y)=Q(x)+v_0(x,y), 
$$ 
where $v_0$ is localized both in $x,y$. More precisely, we suppose that
$v_0\in H^s(\R^2)$ with $s$ large enough. If we set $u=Q+v$ then we have that
$v$ solves the problem
\begin{equation}\label{posledno}
v_{t}-(c-1)v_{x}-v_{txx}+cv_{xxx}+\partial_{x}(v^2)+\partial_{x}(Qv) -
\partial_{x}^{-1}v_{yy}=0,\quad v(0,x,y)=v_0\,.
\end{equation}
In the case $Q=0$ the above equation is shown to be globally well-posed in \cite{SaTz}.
In the case of a $Q$ which is bounded together with its derivatives one needs
to combine the argument of \cite{SaTz} with the following control on the flow
of (\ref{posledno}). Multiplying (\ref{posledno}) by $v$ and integrating over
$\R^2$ yields
$$
\frac{d}{dt}\Big(\|v(t,\cdot)\|_{L^2}^{2}+\|\partial_{x}v(t,\cdot)\|_{L^2}^{2}
\Big)=-2\int \partial_{x}(Qv)v=-\int Q' v^2\,.
$$
A use of the Gronwall lemma provides the control
\begin{equation}\label{krum}
\|v(t,\cdot)\|_{L^2}+\|\partial_{x}v(t,\cdot)\|_{L^2}
\leq (\|v_{0}\|_{L^2}+\|\partial_{x}v_0\|_{L^2})e^{c(1+|t|)}\,.
\end{equation}
The local analysis of \cite{SaTz} shows that in the case $Q=0$ the problem
(\ref{posledno}) is locally well-posed for data such that
$\|v_{0}\|_{L^2}+\|\partial_{x}v_0\|_{L^2}<\infty$.
In order to include the term $\partial_{x}(Qv)$ in the local analysis of
\cite{SaTz} one needs to evaluate the quantity
\begin{equation}\label{quantity}
\|(1-\partial^2_x)^{-1}\partial_{x}(Qv)\|_{L^{1+\varepsilon}_{T}L^2_{x,y}}
+
\|(1-\partial^2_x)^{-1}\partial^{2}_{x}(Qv)\|_{L^{1+\varepsilon}_{T}L^2_{x,y}}
\end{equation}
for some $\varepsilon>0$. The unessential loss $\varepsilon$ (compared to the natural
$L^1_{T}$ coming from the Duhamel formula) is related to the fact that
the well-posedness in \cite{SaTz} is established in Bourgain spaces and the
non-linearity in a Bourgain's norm of type $X^{s,b-1}_{T}$, $b>1/2$ close to $1/2$ can be estimated by
the non-linearity if $L^{1+\varepsilon}_{T}H^s$ with $\varepsilon>0$ close to zero.
But the quantity (\ref{quantity}) can be easily estimated in terms
$\|v\|_{L^{\infty}_{T}L^2}+\|v_x\|_{L^{\infty}_{T}L^2}$ (and even only $\|v\|_{L^{\infty}_{T}L^2}$) which shows that the
term $\partial_{x}(Qv)$ can be incorporated in the local analysis of
\cite{SaTz} which in turn thanks to the control (\ref{krum}) implies that
in the case $p=2$ the equation (\ref{BBM2}) is globally
well-posed for data which is a localized perturbation of $Q$.


\begin{thebibliography}{10}
\bibitem{AGJ} 
{\sc Alexander, J., Gardner, R., Jones, C. }
\newblock A topological invariant arising in the stability analysis of traveling waves.
\newblock {\em J. Reine Angew. Math.   } 410, 167-212 (1990).
%
\bibitem{Ben} {\sc Benjamin, T.},
\newblock The stability of solitary waves,
\newblock {\em Proc. London Math. Soc.} (3) 328, 153-183 (1972).
%
\bibitem{BSS} {\sc Bona, J.L. Souganidis, P. and Strauss W.},
\newblock Stability and instability of solitary waves of Korteweg- de vries type
\newblock {\em Proc. London Math. Soc.} (3) 411, 395-412 (1987).
%
\bibitem{BS}
{\sc Bona, J.L., Sachs, R.L}
\newblock Global Existence of Smooth Solutions and Stability of Solitary Waves
for a Generalized Boussinesq equation.
\newblock {\em Commun. Math. Phys.} 118(1988), 15-29.
%
\bibitem{Bridges}
{\sc Bridges, T.J.}
\newblock Universal geometric conditions for the transverse
 instability of solitary waves.
 \newblock{\em Phys. Rev. Lett.}(12)84(2000), 2614-2617.
%
\bibitem{CL} {\sc T.~Cazenave, P.L.~Lions},
\newblock Orbital stability of standing waves for some nonlinear Schr\"odinger equations,
{\em Comm. Math. Phys.} , 85, 549-561 (1982)
%
\bibitem{Coppel}
{\sc Coppel, W.~A.}
\newblock {\em Dichotomies in stability theory}.
\newblock Springer-Verlag, Berlin, 1978.
\newblock Lecture Notes in Mathematics, Vol. 629.
%
\bibitem{JD}{\sc Dieudonn\'e, J.}, {\it Calcul infinit\'esimal}, Collection M\'ethodes, Hermann Paris, 1980.
%
\bibitem{F}{\sc Fedoriuk, M.}, {\it Metod perevala}, (in russian), Mir, Moscow 1977.
%
\bibitem{GSS}{\sc Grillakis, Shatah J. and Strauss, W.}
\newblock Stability theory of solitary waves in the presence of symmetry II.
\newblock{\em J. Funct. Anal. 94}, 2 (1990), 308--348.
%
\bibitem{GZ}
{\sc Gardner, R.~A., and Zumbrun, K.}
\newblock The gap lemma and geometric criteria for instability of viscous shock
  profiles.
\newblock {\em Comm. Pure Appl. Math. 51}, 7 (1998), 797--855.
%
\bibitem{Grenier}
{\sc Grenier, E.}
\newblock On the nonlinear instability of {E}uler and {P}randtl equations.
\newblock {\em Comm. Pure Appl. Math. 53}, 9 (2000), 1067--1091.
%
\bibitem{GHS}
{\sc Groves, M., Haragus, M. and Sun, S.M.}
\newblock Transverse instability of gravity-caillary line solitary water waves.
\newblock {\em C.R. Acad. Sci. Paris 333}, (2001), 421-426.
%
\bibitem{IK}{\sc Ionescu A. and Kenig, C.},
{\it Local and global well-posedness of periodic {KP-I} equations}, Preprint~2005.
%
\bibitem{Johnson}
{\sc Johnson, R.}
\newblock A two dimensional Boussinesq equation for waves and some of its
solutions.
\newblock {\em J. Fluid Mech.}, 323, 65-78.
%
\bibitem{IN}
{\sc Iorio, R. and Nunes, W.}
\newblock On equations of KP-type.
\newblock {\em Proc. Roy. Soc. Edinburgh}, A 128 (1998), 725-743.
%
\bibitem{KS}
{\sc Kapitula, T., and Sandstede, B.}
\newblock Stability of bright solitary-wave solutions to perturbed nonlinear
  {S}chr\"odinger equations.
\newblock {\em Phys. D 124}, 1-3 (1998), 58--103.
%
\bibitem{Kato}
{\sc Kato, T.}
\newblock {\em Perturbation theory for linear operators}.
\newblock Classics in Mathematics. Springer-Verlag, Berlin, 1995.
\newblock Reprint of the 1980 edition.
%
\bibitem{KP}
{\sc Kato, T. and Ponce, G.}
\newblock Commutator estimates and the Euler and Navier-Stokes equations.
\newblock {\em Comm. Pure Appl. Math.}, 41 (1988), 891-907.
%
\bibitem{Kreiss-Kreiss}
{\sc Kreiss, G., and Kreiss, H.-O.}
\newblock Stability of systems of viscous conservation laws.
\newblock {\em Comm. Pure Appl. Math. 51}, 11-12 (1998), 1397--1424.
%
\bibitem{Liu} {\sc Liu, Y.},
\newblock Strong instability of solitary wave solutions to a Kadomtsev-Petviashvili equation in three dimensions.
\newblock {\em J. Diff. Equations}, (2002), 153-170.
%
\bibitem{Mameri} {\sc Mammeri, Y.},
\newblock Comparaison entres mod\`eles d'ondes de surfaces en dimension $2$.
\newblock {\em M2AN}, (2007), 513-542.

\bibitem{Metivier-Zumbrun}
{\sc M{\'e}tivier, G., and Zumbrun, K.}
\newblock Large viscous boundary layers for noncharacteristic nonlinear
  hyperbolic problems.
\newblock {\em Mem. Amer. Math. Soc. 175}, 826 (2005), vi+107.
%
\bibitem{Molinet}
{\sc Molinet, L.}
\newblock On the asymptotic behavior of solutions to the (generalized)
Kadomtsev-Petviashvili-Burgers equation
\newblock {\em J. Diff. Eq.} , 152, 30-74 (1999)
%
\bibitem{MST}{\sc Molinet, L., Saut, J.-C., and Tzvetkov, N.}
\newblock Global well-posedness for the KP-I equation on the background of a non-localized solution
\newblock {\em Comm. Math. Phys.} , 272, 775-810 (2007)
%
\bibitem{PW} {\sc Pego, R., and Weinstein, M.},
\newblock Eigenvalues, and instabilities of solitary waves. 
\newblock {\em Phil. Trans. R. Soc. London A} 340 (1992), 47-97.
%
\bibitem{RT} {\sc Rousset, F., and Tzvetkov, N.},
\newblock Transverse nonlinear instability for two-dimensional dispersive models.
\newblock {\em Ann. IHP, Analyse non lin\'eaire}, to appear.
%
\bibitem{Saut}{\sc Saut, J.-C.},
\newblock Remarks on the generalized {K}adomtsev- {P}etviashvili equations.
\newblock {\em Indiana Univ. Math. J.} 42 (1993), 1011-1026.
%
\bibitem{SaTz}{\sc Saut, J.-C. and Tzvetkov, N.},
\newblock Global well-posedness of the KP-BBM equations.
\newblock {\em AMRX } (2004), 1-16.
%
\bibitem{T} {\sc Titchmarch, E.C.}, 
{\it Eigenfunction expansions associated to second order differential equations},
Clarendon Press, Oxford, 1946.
%
\bibitem{W} {\sc Weinstein, M.},  
\newblock Modulational stability of ground states of Nonlinear Schr\"odinger
equations, 
\newblock {\em SIAM J. Math. Anal.} 16, 472-491 (1985).
%
\bibitem{ZK} {\sc Zakharov, V.E., and Kuznetsov, E.A.},
\newblock {\em Zh. Eksper.-Teoret. Fiz.}, 66 (1974), 594-597.

\end{thebibliography}
\end{document}